\documentclass[11pt]{article}
\usepackage{amssymb}
\usepackage{amsmath}
\usepackage{authblk}
\usepackage{arydshln}

\def\denom{168}
\newcommand{\W}{2}
\def\lessspacepmod#1{\ (\text{mod}\ #1)}
\newcommand{\WW}{}
\newcommand{\dd}{}
\newcommand{\1}{(1)}
\newcommand{\2}{(2)}
\usepackage{bbm}
\usepackage{amsthm}
\usepackage{mathtools}
\usepackage{lineno}
\newtheorem{theorem}{Theorem}
\newtheorem{lemma}{Lemma}
\newtheorem{proposition}{Proposition}

\theoremstyle{remark}

\title{Carmichael Numbers in All Possible Arithmetic Progressions}
\author{Daniel Larsen}
\date{}
\begin{document}
\maketitle
\begin{abstract}
We prove that every arithmetic progression either contains infinitely many Carmichael numbers or none at all. Furthermore, there is a simple criterion for determining which category a given arithmetic progression falls into. In particular, if $m$ is any integer such that $(m,2\phi(m))=1$ then there exist infinitely many Carmichael numbers divisible by $m$. As a consequence, we are able to prove that $\liminf_{n\text{ Carmichael}}\frac{\phi(n)}{n}=0$, resolving a question of Alford, Granville, and Pomerance \cite{AGP:1994}. 
\end{abstract}
\section{Introduction}
If $n$ is a positive integer such that $a^n\equiv a\lessspacepmod{n}$ for every integer $a$ then the Fermat primality test predicts that $n$ is prime. Nevertheless, there exist composite values of $n$ for which this is also true. Such values are called \emph{Carmichael numbers}, though it was \v{S}imerka \cite{S:1885} who discovered the first known example, $561$. According to Korselt's criterion \cite{K:1899}, a square-free composite number $n$ is a Carmichael number if and only if for every prime $p$ dividing $n$, $p-1$ divides $n-1$.

In 1994, the seminal work of Alford, Granville, and Pomerance \cite{AGP:1994} showed that there are infinitely many Carmichael numbers. In analogy with the study of primes, a natural subsequent line of inquiry is the distribution of Carmichael numbers in arithmetic progressions. Wright \cite{W:2013} (following Banks-Pomerance \cite{BP:2010} and Matom\"{a}ki \cite{M:2013}) has shown that there are infinitely many Carmichael numbers which are $a$ mod $q$ when $(a,q)=1$. Unlike in the case of primes, the condition $(a,q)>1$ does not immediately rule out the possibility of infinitely many Carmichael numbers which are $a$ mod $q$. A simpler question, raised by Alford, Granville, and Pomerance, is whether there exists a fixed integer $m>1$ which divides infinitely many Carmichael numbers. Although Banks \cite{B:2021} has called it ``an age-old question,'' we do not know an earlier reference for it than \cite{AGP:1994}. Our main theorem answers this question, proving that every divisor of one Carmichael number divides infinitely many of them. More generally, we show that arithmetic progressions for which Carmichael numbers are not excluded by elementary modular constraints in fact contain infinitely many. The precise statement will be given later in the introduction as Theorem \ref{main}.

We will now present a rough sketch of the Alford-Granville-Pomerance (AGP) method, both to establish notation and to explain why, in its existing form, it cannot establish the infinitude of Carmichael numbers which are divisible by a fixed integer $m>1$. In \cite{AGP:1994}, the authors fix a parameter $y$ and set $\mathcal{Q}$ to be the set of primes $q\leq y$ with the property that $q-1$ has no large prime factors, meaning that if $p\mid q-1$ then $p\leq y^{1-\epsilon}$. Then they set $L=\prod_{q\in\mathcal{Q}}q$. The significance of this construction is that the least common multiple $\lambda(L)$ of the $q-1$ is very small compared to $e^y$, which makes $(\mathbb{Z}/L\mathbb{Z})^\times$ easy to work with later on. For each $d\mid L$, they consider the set of $k\leq x$ for which $dk+1$ is prime, where $x$ is some parameter which is roughly speaking exponential in $y$. Using results on primes in arithmetic progressions, it can be shown that about as many such values of $k$ exist as would be expected, which is to say about $\frac{x}{\log x}$. Therefore, this construction gives many pairs $(d,k)$ with $d\mid L$ and $k\leq x$ such that $dk+1$ is prime, and from this, we can infer that there is some $k$ in particular for which there are many $d\mid L$ with $dk+1$ prime. Let the set of these primes be $\mathcal{P}$.

We will now use these primes to construct a Carmichael number. Suppose we can find a product $\Pi$ of distinct elements of $\mathcal{P}$ which is congruent to 1 mod $L$. We claim that $\Pi$ is then necessarily a Carmichael number. Indeed, if $p\mid \Pi$ then $p$ takes the form $dk+1$, and since $d\mid L$, it suffices to show that $\Pi\equiv 1\lessspacepmod{kL}$. By assumption, $\Pi\equiv 1\lessspacepmod{L}$ and by construction, $\Pi\equiv 1\lessspacepmod{k}$. Therefore, assuming that $(k,L)=1$ (which can be arranged), we conclude from Korselt's criterion that $\Pi$ is indeed a Carmichael number.

A key part of the argument is showing that one can find products of elements of $\mathcal{P}$ which are congruent to 1 mod $L$. Roughly speaking, one can show using character theory that the residues of products of distinct elements of $\mathcal{P}$ equidistribute inside the smallest subgroup of $(\mathbb{Z}/L\mathbb{Z})^{\times}$ which contains almost all of the elements of $\mathcal{P}$. This proof makes critical use of the fact that $(\mathbb{Z}/L\mathbb{Z})^{\times}$ has no large cyclic factors, in the form that if $\chi$ is a Dirichlet character mod $L$ and $\chi(n)\neq 1$ then $\vert\chi(n)-1\vert$ is relatively large. One would expect that no proper subgroup of $(\mathbb{Z}/L\mathbb{Z})^{\times}$ would contain almost all the elements of $\mathcal{P}$. However, this is difficult to prove and unnecessary for Alford, Granville, and Pomerance; in fact, the existence of such a subgroup would increase the number of Carmichael numbers produced by their method.

The most obvious way to show that there are infinitely many Carmichael numbers divisible by $m$ would be to mimic the AGP construction and then multiply by $m$ at the end. For instance, one could seek a product $\Pi$ of elements in $\mathcal{P}$ which is congruent to $\frac{1}{m}$ mod $L$. There is the potential difficulty that the elements of $\mathcal{P}$ might belong to a subgroup of $(\mathbb{Z}/L\mathbb{Z})^{\times}$ which does not contain $\frac{1}{m}$. This seems rather unlikely, however, so one could hope that probabilistic methods would establish this to be non-generic behavior. A more fundamental problem is that while $m\Pi$ might be congruent to 1 mod $L$, it will be congruent to $m$ mod $k$. Therefore, the AGP method cannot be used directly to show that there are infinitely many Carmichael numbers divisible by $m$.

In this paper, we use a modification of their method, the main idea of which is to use two sets of primes of the form $\{dk+1\}$ instead of one. More explicitly, for suitable integers $k_1$, $k_2$, $L_1=\prod_{q\in\mathcal{Q}_1}q$, and $L_2=\prod_{q\in\mathcal{Q}_2}q$, we take $$\mathcal{P}_1:=\{dk_1+1: d\in \mathcal{D}_1\}$$ and $$\mathcal{P}_2:=\{dk_2+1: d\in\mathcal{D}_2\},$$ where $\mathcal{D}_1$ and $\mathcal{D}_2$ are drawn from the divisors of $L_1$ and $L_2$, respectively. Our goal is to find a square-free integer $\Pi_1$ whose prime factors lie in $\mathcal{P}_1$ such that $\Pi_1\equiv 1\lessspacepmod{L_1}$ and $\Pi_1\equiv \frac{1}{m}\lessspacepmod{k_2L_2}$. If this is in general possible then we can swap indices and obtain $\Pi_2$ with corresponding properties. Assuming that $k_1$, $k_2$, $L_1$, and $L_2$ are mutually coprime, $m\Pi_1\Pi_2\equiv 1\lessspacepmod{k_1k_2L_1L_2}$. If we can also arrange that $\Pi_1$ and $\Pi_2$ are 1 mod $\phi(m)$, Korselt's criterion will be satisfied. This last condition is not very serious. In particular, the condition is trivial in the case $m=3$, which is already sufficient to solve a problem posed by Alford, Granville, and Pomerance \cite{AGP:1994} as to whether $\frac{\phi(n)}{n}\to 1$ as $n$ ranges over Carmichael numbers.

The construction of $\mathcal{Q}_1$ and $\mathcal{Q}_2$, as well as the development of the necessary sieve methods, takes up Sections \ref{AP} and \ref{set-up}. In Section \ref{k-constr}, we show that there are many values of $k_1$ and $k_2$ for which the associated sets $\mathcal{P}_1$ and $\mathcal{P}_2$ are large. In Section \ref{carmichaelconstruction}, we construct auxiliary products $A_i$ which are 1 mod $L_i$ and $\frac{1}{m}$ mod $L_{3-i}$, for $i\in\{1,2\}$. Because the methods in that section preview those in Section \ref{extra}, we will say nothing more of the construction of $A_1$ and $A_2$ for now. We will instead give an outline of the strategy of Sections \ref{heart} and \ref{extra}, in which we finish the construction of $\Pi_1$ and $\Pi_2$.

We would like to use character methods to show that there are many products $\Pi'_1$ of distinct elements of $\mathcal{P}_1$ which are $1$ mod $L_1L_2$ and $(mA_1)^{-1}$ mod $k_2$. If we can show this then $\Pi_1$ can be taken to be $A_1\Pi'_1$ (and $\Pi_2$ can be dealt with analogously). The idea is to show that for most choices of $k_1$ and $k_2$, it holds for any non-trivial character $\chi$ of $(\mathbb{Z}/k_2L_1L_2)^\times$ that the average value of $\chi(\Pi_1)$, as $\Pi'_1$ ranges over the set of products of $\mathcal{P}_1$, is extremely close to $0$. To do this, it suffices to show that $\chi$ does not take values extremely close to $1$ on almost all elements of $\mathcal{P}_1$. Suppose $\chi$ is a counterexample. Let $L=L_1L_2$. If $\chi^{\lambda(L)}$ is not trivial then we replace $\chi$ with $\chi^{\lambda(L)}$, as $\lambda(L)$ is by construction sufficiently small that the values of this new character will still be very close to $1$. After this potential replacement, we could hope that the resulting $\chi$ would be trivial either on $(\mathbb{Z}/L\mathbb{Z})^\times$ or on $(\mathbb{Z}/k_2\mathbb{Z})^\times$. This would be true if the mod $L$ component and the mod $k_2$ component of the character had relatively prime order. By careful construction of $k_2$, we can almost ensure that this happens, but it is possible that $\chi$ will be merely quadratic mod $L$ or $k_2$, without being trivial.

Some effort will have to be spent closing this hole, but for now, we will simply assume $\chi$ really is trivial mod $L$ or $k_2$. We can avoid the former case by ruling out the possibility that almost all $p\in \mathcal{P}_1$ lie in an index $r$-subgroup of $(\mathbb{Z}/L_2\mathbb{Z})^\times$ for any particular prime $r$. The key idea is that if we fix $r$, we can work modulo the product of all primes in $\mathcal{Q}_2$ which are $1$ mod $r$. This is a small enough modulus that we can hope to prove equidistribution of the elements of $\mathcal{D}_1$ and therefore of the elements of $\mathcal{P}_1$. This equidistribution is actually established in Section \ref{carmichaelconstruction}, where it is more essential. Indeed, because $\Pi'_1$ is meant to be 1 mod $L$, it isn't a problem if our primes are contained in some proper subgroup mod $L$; the point at which equidistribution is really important is in the construction of $A_1$, where we are trying to find a product congruent to $\frac{1}{m}$ mod $L_2$.

Our choice of $k_1$ and $k_2$ is aimed at dealing with the second of the two cases, that $\chi(p)$ depends only on $p$ mod $k_2$. Unless the order of $\chi$ is quite large, we only need to rule out the possibility that mod $k_2$ almost all $p\in \mathcal{P}_1$ lie in $\ker \chi$. In the case that the order of $\chi$ is quite small, we use a version of the large sieve which deals with characters of fixed prime order to show that such a phenomenon happens for a relatively small number of choices for $k_1$ and $k_2$. For this sieve result, we need to use some algebraic number theory, in the form of $r$th power reciprocity. On the other hand, if $\chi$ is of larger order then the character values being very close to 1 becomes unlikely, to the extent that this happening for most values of $k_1$ and $k_2$ is combinatorially impossible. We conclude that for generic choices of $k_1$ and $k_2$, the elements of $\mathcal{P}_1$ will not have character values clustered around the identity. Then the products of the primes will equidistribute mod $k_2L$, allowing us to choose $\Pi'_1$ satisfying the desired congruences.

So far, we have only discussed how to construct Carmichael numbers congruent to 0 mod $m$. Dealing with general arithmetic progressions will be only slightly more involved. To demonstrate, we turn to arithmetic progressions in the opposite extreme: those corresponding to elements of $(\mathbb{Z}/m\mathbb{Z})^{\times}$. If $r$ is a prime which does not divide $m$ then Carmichael numbers congruent to $r$ mod $m$ can be obtained by constructing Carmichael numbers divisible by $r$ with the additional property that each of the added factors is 1 mod $m$. (This could be arranged by taking $k_1$ and $k_2$ to be divisible by $m$; however, this approach introduces some additional complications that we are able to avoid with a slightly different method.) As mentioned previously, these arithmetic progressions are already known to contain infinitely many Carmichael numbers. A long-standing problem has been to show that there are at least $x^{\epsilon}$ Carmichael numbers in such arithmetic progressions for $x$ sufficiently large in terms of $\epsilon$ (a positive constant). When $r$ is a quadratic residue mod $m$, Matom\"{a}ki has established this. However, when $r$ is not a quadratic residue, the best known bound is $x^{\frac{1}{6\log\log\log x}}$, due to Pomerance \cite{P:2021}. This reflects a fundamental limitation of Wright's method \cite{W:2013} which involves 2-subgroups of $(\mathbb{Z}/mL\mathbb{Z})^{\times}$. Our method is based on a different approach and allows us to prove that there exist at least $x^{1/\denom-\epsilon}$ Carmichael numbers congruent to $r$ mod $m$ for sufficiently large $x$ for all $r$ and $m$ for which modular considerations do not forbid Carmichael numbers. This gives a definitive solution to a suitably modified version of a conjecture of Banks and Pomerance \cite{BP:2010}.

Let $r\lessspacepmod{m}$ be an arithmetic progression. Set $g=(r,m)$ and then letting $\lambda(g)$ be the maximum element order in $(\mathbb{Z}/g\mathbb{Z})^{\times}$, set $h=(\lambda(g),m)$. We say that $r\lessspacepmod{m}$ is \emph{Carmichael incompatible} if any of the following occur:

\begin{itemize}
\item $(g,2\phi(g))>1$
\item $h\nmid r-1$
\item $36\mid m$, $r\equiv 3\lessspacepmod{12}$, and $r/g\equiv 5\text{ or }7\lessspacepmod{12}$
\end{itemize}

Otherwise, we say $r\lessspacepmod{m}$ is \emph{Carmichael compatible}. We will show in Section \ref{end} that only Carmichael compatible arithmetic progressions can contain Carmichael numbers. On the other hand, the main theorem we will prove in this paper is that every Carmichael compatible arithmetic progressions contains infinitely many Carmichael numbers.

\begin{theorem}\label{main}
Let $r\lessspacepmod{m}$ be a Carmichael compatible arithmetic progression. Then for every $\epsilon>0$ and for every $x$ sufficiently large in terms of $\epsilon$, there are more than $x^{1/\denom-\epsilon}$ Carmichael numbers less than $x$ congruent to $r$ mod $m$.
\end{theorem}

We will show in Section \ref{end} that this theorem is a consequence of the following proposition.

\begin{proposition}\label{workhorse}
Let $m_0$ be an odd square-free integer, and let $\ell_0$ be a positive integer which is not divisible by $4$. Let $a_0$ mod $q_0$ be a relatively prime residue with $(a_0-1,q_0)=2$. Then for every $\epsilon>0$ and every $x$ sufficiently large in terms of $\epsilon$, there exist more than $x^{\frac{1}{\denom}-\epsilon}$ values of $\Pi\leq x$ with $\Pi$ square-free and relatively prime to $2m_0$, with $\omega(\Pi)\equiv 2\lessspacepmod{\ell_0}$, and such that for $p\mid \Pi$, we have $\frac{p-1}{2}\mid m_0\Pi-1$ and $p\equiv a_0\lessspacepmod{q_0}$.
\end{proposition}

For Sections \ref{set-up}--\ref{extra}, we will fix $a_0$, $q_0$, $m_0$, and $\ell_0$ satisfying the conditions of this proposition. Then we will construct values of $\Pi$ satisfying the desired congruences, values we will call \emph{Carmichael seeds}.

The constant $1/\denom$ is not the optimal value which can be obtained by this method. To begin with, we expect that more involved sieve arguments could yield a slightly stronger bound.

In a recent preprint \cite{W:2024}, Wright independently introduced the idea of using two sets of primes to resolve, under the assumption of a strong conjecture on the smallest prime in an arithmetic progression, the question of whether there exist infinitely many Carmichael numbers $\prod_{i=1}^m p_i$ for which $\gcd(p_1-1,\ldots,p_m-1)$ is bounded.  Our main theorem immediately resolves this question unconditionally.

Alford, Granville, and Pomerance ask about numbers which satisfy an altered version of Korselt's criterion in which $p-a\mid n-b$. To our knowledge, the infinitude of composite numbers with this property is an open question unless $a,b\in\{\pm 1\}$. (See \cite{W:2018} for a discussion of why this is the case.) It seems the tools of this paper should be able to resolve this question for $a\in\{\pm 1\}$ and $b\in 2\mathbb{Z}+1$.

Finally, these methods should also be applicable to Fermat pseudoprimes, which can be found in arithmetic progressions which do not contain any Carmichael numbers. 
\section{Notation}
We use $\mathbb{P}$ to refer to the set of primes. The von Mangoldt function $\Lambda(n)$ is defined to be $\log p$ if $n=p^m$ for a prime $p$ and 0 otherwise. We use $\pi(x)$ to represent the number of primes up to $x$, while $\psi(x)$ gives a weighted count: $$\psi(x)=\sum_{n\leq x}\Lambda(n).$$ Similarly, $$\pi(x; q,a)=\#\{p\leq x: p\equiv a\lessspacepmod{q}\}$$ and $$\psi(x; q,a)=\sum_{\substack{n\leq x\\ n\equiv a\lessspacepmod{q}}}\Lambda(n).$$ We also let $$\psi(x,\chi)=\sum_{n\leq x}\Lambda(n)\chi(n)$$ and $$\psi'(x,\chi)=\begin{cases}\psi(x,\chi)&\chi\text{ non-principal}\\ \psi(x,\chi)-x&\text{else}\end{cases}$$ We use $(m,n)$ and $[m,n]$ to denote, respectively, the greatest common divisor and the least common multiple of $m$ and $n$. We take $\phi(n)$ to be Euler's totient function, defined as the order of the unit group of $\mathbb{Z}/n\mathbb{Z}$. Another important arithmetic function will be the Carmichael $\lambda$ function, defined as the exponent of $(\mathbb{Z}/n\mathbb{Z})^{\times}$, that is to say, the maximal element order in this group. If $n=\prod_{i=1}^kp_i$ is an odd square-free number then $\lambda(n)$ is the least common multiple of $\{p_i-1\}_{1\leq i\leq k}$. We define $\omega(n)$ to be the number of distinct primes dividing $n$. We define $\text{P}^-(n)$ to be the smallest prime factor of $n$. Two more important functions will be $\text{exp}(x):=e^x$ and $e(x):=e^{2\pi ix}$.

If we want to sum over primitive characters of $(\mathbb{Z}/n\mathbb{Z})^{\times}$, we will write $\sum_{\chi\text{ mod }n}^{\star}$. We will let $\chi_0$ be the trivial character, with the group $\chi_0$ is defined on implicit. If there is a character $\chi$ being discussed, $\chi_1$ will be taken to be the primitive character which induces $\chi$. The conductor of $\chi_1$ will be called $q_1$. In general, $p$ will always be a prime, whereas $q$ will sometimes represent a prime and in other places will represent the conductor of a character or more generally an integer such that $(\mathbb{Z}/q\mathbb{Z})^{\times}$ is important in some way. (Which of the two is meant should be clear from context.) 

We use Vinogradov notation in the standard fashion: $f\ll g$ means that there exists a constant $C$ with $f(x)\leq Cg(x)$ for sufficiently large $x$, while $f\gg g$ means there exists a constant $c>0$ with $f(x)\geq cg(x)$ for sufficiently large $x$. We also use $O$ notation, so $f=O(g)$ means $f\ll g$ while $f=o(g)$ means that the ratio of $\frac{f(x)}{g(x)}$ goes to 0. This has an additional flexibility which cannot be handled with Vinogradov notation. For instance, we can use $O$ notation in exponents. In this case, we may use a sign, as in the expression (occurring quite frequently in the paper) $x^{1-o(1)}$, which means a function $f(x)$ such that for all $\epsilon>0$, $f(x)>x^{1-\epsilon}$ for sufficiently large $x$ in terms of $\epsilon$. In general, we will assume that variables with the name $x$, $y$, or $N$ are large, and our asymptotic expressions will be premised on these variables being sufficiently large. Another useful convention will be to write $x\sim y$ for $x\in [y,2y)$. If we say that $x$ is bounded by $y$ then we mean $x\leq y$, while if we say that $x$ is bounded asymptotically by $y$ then we mean $x=O(y)$.

We say that almost none of the elements in a set $\mathcal{S}$ have a property if the number of elements with the property is $o(\vert\mathcal{S}\vert)$. We say that almost all of the elements satisfy a property if almost none of the elements satisfy the negation of the property.

Because this paper is based around an extended construction in which many variables and parameters are used, we will devote the remainder of this section to collecting their values (or approximate values) in one place. The starting parameters are $y$ and $\iota$. The former determines the size of the Carmichael numbers and the latter determines how small the $\epsilon$ in Theorem \ref{main} is.  We will choose $\iota$ to be very small and $y$ to be very large with reference to $\iota$. It will be convenient to assume $1/\iota$ is an integer. Beginning at the end of Section \ref{AP}, we will set $\delta$ to $\frac{1}{6}$. 
We also take $\theta=\frac{1}{6}-2\iota$ and $\rho=\frac{1}{24}-2\iota$. Finally, we take $\kappa$ and $T$ to be large integers; for instance, they can both safely be given the value $100$.

Now we discuss the orders of magnitude of the various size parameters that will appear. We will set $x=y^{\frac{w}{\delta-3\iota}}$ where $w=\lfloor y^{\theta+\iota\rho}\rfloor$. The important thing to remember is that $x$ is much larger than $\lambda(L)$ (while being much smaller than $L$). In particular, we will set $J=\lambda(L)^{10}$, and $J$ will still be much smaller than $x$. Roughly, $\mathcal{D}_i$ will be the set of divisors of $L_i$ with $w$ prime factors.

\section{Primes in Arithmetic Progressions}\label{AP}
Before beginning the proof of our main theorem, we devote a section to the distribution of primes in arithmetic progressions, which will be relevant at several points in the construction. All the results in this section are similar to existing results in the literature, but it is useful to have versions of these statements in the exact form that we will want them. Let $\mathcal{A}$ be a set of positive integers less than $x$ and $\mathcal{P}$ be a set of primes. We set $$S(\mathcal{A},\mathcal{P},z)=\#\{a\in\mathcal{A}: (a,\mathcal{P}(z))=1\}$$ where $$\mathcal{P}(z)=\prod_{\substack{p\in\mathcal{P}\\ p\leq z}}p.$$ We also set $\mathcal{A}_d=\{a\in\mathcal{A}: d\mid a\}$. Given a multiplicative function $g$, we set $$R_d=\vert \mathcal{A}_d\vert-\frac{\vert\mathcal{A}\vert}{g(d)}$$ for $d$ square-free. We say that $\mathcal{A}$ has level of distribution $\theta$ with respect to $g$ and $\mathcal{P}$ if $$\sum_{\substack{d\leq x^{\theta}\\ d\mid \mathcal{P}(z)}}R_d\ll \frac{x}{\log^4 x}.$$ Let $\iota$ be a small positive constant. The asymptotic bounds in this section will be allowed to depend on $\iota$.

\begin{lemma}\label{realrosser}
Suppose $\vert\mathcal{A}\vert\gg \frac{x}{\log^2 x}$. Let $\mathcal{P}'=\mathbb{P}\setminus \mathcal{P}$ and $\mathcal{P}^{(1)}\coprod\mathcal{P}^{(2)}$
be a partition of $\mathcal{P}$. We assume that $2\in\mathcal{P}'$.
Let $g$ be the multiplicative function defined by $g(p)=p$ for $p\in\mathcal{P}^{(2)}$ and $g(p)=p-1$ for all other primes. If $\mathcal{A}$ has level of distribution $(2+\iota)\delta$ with respect to $g$ and $\mathcal{P}$ then $$S(\mathcal{A},\mathcal{P},x^{\delta})\gg \prod_{\substack{p\in\mathcal{P}'\\ 2<p\leq x^{\delta}}}\Big(1-\frac{1}{p-1}\Big)^{-1}\frac{\vert \mathcal{A}\vert}{\log x}$$
\end{lemma}
\begin{proof}
This is a consequence of Theorem 1 of \cite{I:1980}, taking $s=2+\iota$, $\omega(p)=\frac{p}{g(p)}$, and $\kappa=1$. We set $$U(z)=\prod_{p\mid \mathcal{P}(z)}\Big(1-\frac{1}{p-1}\Big)$$ and $$V(z)=\prod_{p\mid \mathcal{P}(z)}\Big(1-\frac{1}{g(p)}\Big).$$ 
Now, $$\log z\prod_{\substack{p\mid\mathcal{P}'(z)\\ p>2}}\Big(1-\frac{1}{p-1}\Big)U(z)=\log z\prod_{2<p\leq z}\Big(1-\frac{1}{p-1}\Big)$$ approaches some constant as $z\to\infty$. Therefore, $$\vert\mathcal{A}\vert V(z)
\geq \vert\mathcal{A}\vert U(z)
\gg  \prod_{\substack{p\mid\mathcal{P}'(z)\\ p>2}}\Big(1-\frac{1}{p-1}\Big)^{-1}\frac{\vert\mathcal{A}\vert}{\log x}$$ for $z=x^{\delta}$. This shows that the claimed lower bound in this lemma holds when considering solely the main term in Theorem 1 of \cite{I:1980}. Since the error term is $$\sum_{\substack{d<x^{(2+\iota)\delta}\\ d\mid \mathcal{P}(x^{\delta})}}R_d,$$ the lemma follows by using the level of distribution of $\mathcal{A}$.
\end{proof}
\begin{lemma}\label{rosser}
Suppose $\delta$ and $\epsilon$ are positive constants. Suppose $d\leq x^{\delta}$ is an even number such that \begin{equation}\label{bveq}\sum_{\substack{q\leq x^{\delta+2(1+\iota)\epsilon}\\ d\mid q}}\max_{(a,q)=1}\Big\vert\pi(dx; q,a)-\frac{\pi(dx)}{\phi(q)}\Big\vert\leq \frac{x}{\log^5x}.\end{equation} Suppose $b$ and $r\ll \log x$ are integers such that $$(bd+1,r)=(b,r)=1.$$ Finally, suppose $\phi(d)\gg d$. Then $$\#\{k\leq x: k\equiv b\lessspacepmod{r}, dk+1\in\mathbb{P}, \emph{P}^-(k)\geq x^{\epsilon}\}\gg \frac{x}{\phi(r)\log^2 x}.$$
\end{lemma}
\begin{proof}
We may assume $r$ is even and $b$ is odd.
Let $$\mathcal{A}=\{k<x: k\equiv b\lessspacepmod{r}, dk+1\in\mathbb{P}\},$$ $$\mathcal{P}^{(1)}=\{p: p\nmid dr\},$$ 
$$\mathcal{P}^{(2)}=\{p: p\mid d, p\nmid r\},$$ and $$\mathcal{P}'=\{p:p\mid r\}.$$ If $e$ is relatively prime to $r$ then $$R_e=\pi(dx; der, a)-\frac{\phi(d)}{\phi(de)}\pi(dx; dr, bd+1)$$ where $a\equiv bd+1\lessspacepmod{dr}$ and $a\equiv 1\lessspacepmod{e}$.

Therefore, we can bound $\vert R_e\vert$ by $$\max_{(a,der)=1}\Big\vert\pi(dx; der,a)-\frac{\pi(dx)}{\phi(der)}\Big\vert+\frac{\phi(d)}{\phi(de)}\max_{(a,dr)=1}\Big\vert\pi(dx; dr,a)-\frac{\pi(dx)}{\phi(dr)}\Big\vert,$$ 
where we have used the identity 
$$\frac 1{\phi(der)} = \frac{\phi(d)}{\phi(de)\phi(dr)}$$
which follows from $(e,r)=1$.
Consequently, using \eqref{bveq},
$$\sum_{\substack{e<x^{(2+\iota)\epsilon}\\ e\mid \mathcal{P}(x^{\epsilon})}}\vert R_e\vert\leq \frac{x}{\log^5 x}+\frac{x}{\log^5 x}\sum_{e<x^{(2+\iota)\epsilon}}\frac{\mu^2(e)}{\phi(e)}\ll \frac{x}{\log^4x}.$$ Now, $$\Big\vert\pi(dx; dr,db+1)-\frac{\pi(dx)}{\phi(dr)}\Big\vert\leq \frac{x}{\log^5 x}.$$ Since $$\frac{\pi(dx)}{\phi(dr)}\gg \frac{x}{\phi(r)\log x}\gg \frac{x}{\log^2 x},$$ this means that $\vert\mathcal{A}\vert\gg \frac{x}{\phi(r)\log x}$. By Lemma \ref{realrosser}, \begin{align*}\#\{k\leq x: k\equiv b\lessspacepmod{r}, dk+1\in\mathbb{P}, \text{P}^-(k)\geq x^{\epsilon}\}&\gg \frac{\vert\mathcal{A}\vert}{\log x}\\&\gg \frac{x}{\phi(r)\log^2 x}.\end{align*}

\end{proof}
As is usual in sieve theory, getting a good upper bound is much easier.
\begin{lemma}\label{Selberg}
For all $d$, $$\#\{k\leq x: dk+1\in\mathbb{P}, \emph{P}^-(k)\geq x^{\iota}\}\ll \prod_{\substack{p\mid d\\ p>2}}\Big(1-\frac{2}{p}\Big)^{-1}\frac{x}{\log^2 x}.$$
\end{lemma}
\begin{proof}
We assume that $d$ is even, as otherwise, this bound is trivial. Let $\mathcal{A}=\{k(dk+1): k\leq x\}$ and $\mathcal{P}=\mathbb{P}\setminus\mathcal{P}'$ where $\mathcal{P}'=\{p: p\mid d\}$. Now $$\mathcal{A}_e=\{k(dk+1): k\leq x, k\equiv 0\text{ or }-d^{-1}\lessspacepmod{e}\}$$ so $$\Big\vert\vert\mathcal{A}_e\vert-\frac{2x}{e}\Big\vert\leq 2.$$ The Selberg sieve (for instance, Theorem 7.14 of \cite{FI:2010}) tells us that $$S(\mathcal{A},\mathcal{P},z)\leq \frac{\vert\mathcal{A}\vert}{G(z)}+O\Big(\frac{z^2}{\log^2 z}\Big)$$ where $$G(z)=\sum_{\substack{n\mid \mathcal{P}(z)\\ n\leq z}}\prod_{p\mid n}\frac{2}{p-2}.$$ With some manipulation, one finds that $$G(z)\gg \prod_{\substack{p\mid d\\ p>2}}\frac{p-2}{p}\log^2 z.$$ (See, for instance, (7.141) of \cite{FI:2010}.) Thus $$S(\mathcal{A},\mathcal{P},z)\ll \prod_{\substack{p\mid d\\ p>2}}\Big(1-\frac{2}{p}\Big)^{-1}\frac{x}{\log^2 x}$$ for $z=x^{\iota}$.
\end{proof}
Our following three lemmas give bounds for a certain expression which is closely related to the Bombieri-Vinogradov theorem.
\begin{lemma}\label{bv1}
$$\sum_{q\leq x^{1/6}}\frac{q}{\phi(q)}\sum_{\chi\emph{ mod }q}^{\star}\max_{x'\leq x}\vert \psi'(x', \chi)\vert\ll x\log^4 x.$$
\end{lemma}
\begin{proof}
This is an immediate consequence of Vaughan's inequality. (See equation (2) of Chapter 28 of \cite{D:1980}.)
\end{proof}
\begin{lemma}\label{bvugh}
For all $A$, and for sufficiently large $x$ in terms of $A$, $$\sum_{q\leq \log^A x}\frac{q}{\phi(q)}\sum_{\chi\emph{ mod }q}^{\star}\max_{x'\leq x}\vert \psi'(x', \chi)\vert\ll \frac{x}{\log^A x}.$$
\end{lemma}
\begin{proof}
This follows immediately from the Siegel-Walfisz theorem (for instance, equation (4) of Chapter 22 of \cite{D:1980}).
\end{proof}
\begin{lemma}\label{bv2}
For some absolute constant $c>0$, we can associate to every integer $x$ an integer $s(x)$ with the property that $$\sum_{\substack{q\leq e^{c\sqrt{\log x}}\\ s(x)\nmid q}}\frac{q}{\phi(q)}\sum_{\chi\emph{ mod }q}^{\star}\max_{x'\in [\sqrt x,x]}
\frac{\vert \psi'(x', \chi)\vert}{x'}\leq \frac{1}{e^{c\sqrt{\log x}}}$$ 
and with $s(x)$ greater than $\log^A x$ for $x$ sufficiently large in terms of $A$.
\end{lemma}
\begin{proof}
This is very closely related to Lemma 11.2 of \cite{LP:2019}. There is some constant $c'>0$ such that for at most one primitive character $\chi$ whose modulus is bounded by $e^{\sqrt{\log x}}$, the associated $L$-function $L(z,\chi)$ has a real zero $\beta>1-\frac{c'}{\sqrt{\log x}}$. If such a character exists, $s(x)$ is defined to be its modulus. We assume that such a character does exist, as otherwise we need not worry about $s(x)$ at all. Siegel's theorem \cite{S:1935} tells us that for every $\epsilon>0$, there exists a constant $C(\epsilon)$ such that $\beta\leq 1-C(\epsilon)s(x)^{-\epsilon}$. We then have $\frac{\sqrt{\log x}}{c'}<\frac{s(x)^{\epsilon}}{C(\epsilon)}$. Take $\epsilon=\frac{1}{2A+2}$. Then $$s(x)>\Big(\frac{C(\frac{1}{2A+2})}{c'}\Big)^{2A+2}\log^{A+1} x>\log^A x$$ when $x$ is large.

Now,  for some absolute positive constant $c$, 
$$\frac{|\psi'(x',\chi)|}{x'}\ll e^{-4c\sqrt{\log x}}$$
if $q<e^{\sqrt{\log x}}$ and $s(x)\nmid q$; this is (11.3) of \cite{LP:2019}. The result now follows.
\end{proof}

Note that the previous two lemmas are ineffective when it comes to giving a specific lower bound for $x$ in terms of $A$. Let $\delta=\frac{1}{6}$. We then have the following version of the Bombieri-Vinogradov theorem. 
\begin{proposition}\label{bv3}
Let $N$ be a square-free number whose prime factors are bounded by $(\log x)^{\gamma}$ for some constant $\gamma>1$. Let $$\mathcal{D}=\{d\mid N: \omega(d)=w\},$$ for some positive integer $w<\omega(N)/2$. Let $t=x^{\delta-3\iota}$. We assume that $$(s(\lfloor tx\rfloor),N)<\frac{s(\lfloor tx\rfloor)}{\log^{2A}x}$$ for some integer $A$ and that every element of $\mathcal{D}$ is less than $t$. Then the value of $$\sum_{\substack{q\leq x^{\delta}\\ \W d\mid q}}\max_{(a,q)=1}\Big\vert\pi(dx; q,a)-\frac{\pi(dx)}{\phi(q)}\Big\vert$$ is less than $\frac{x}{\log^A x}$ for almost all $d\in\mathcal{D}$, assuming $x$ is sufficiently large in terms of $A$ and $\gamma$.
\end{proposition}
\begin{proof}
We abbreviate $s(\lfloor tx\rfloor)$ as just $s$. Our goal will be to show that $$\sum_{d\in \mathcal{D}}\sum_{\substack{q\leq x^{\delta}\\ \W d\mid q}}\max_{(a,q)=1}\Big\vert\psi(dx; q,a)-\frac{dx}{\phi(q)}\Big\vert$$ is small. Using standard manipulations, this expression is asymptotically bounded by $$\WW\sum_{d\in \mathcal{D}}\dd\sum_{\substack{q\leq x^{\delta}\\ \W d\mid q}}\frac{1}{\phi(q)}\sum_{\chi\text{ mod }q}\vert \psi'(dx,\chi_1)\vert$$ plus an error term (which comes from switching from $\chi$ to its primitive counterpart $\chi_1$)
of $\WW \vert \mathcal{D}\vert x^{\delta}\log^2 x$. Now, 
\begin{align*}\sum_{d\in \mathcal{D}}\dd\sum_{\substack{q\leq x^{\delta}\\ \W d\mid q}}\frac{1}{\phi(q)}&\sum_{\chi\text{ mod }q}\vert \psi'(dx, \chi_1)\vert\\&=\sum_{d\in \mathcal{D}}\dd\sum_{q_1}\sum_{\substack{q\leq x^{\delta}\\ [q_1,\W d]\mid q}}\frac{1}{\phi(q)}\sum_{\chi\text{ mod }q_1}^{\star}\vert \psi'(dx, \chi)\vert\\&\leq \sum_{d\in \mathcal{D}}\dd\sum_{q_1}\frac{1}{\phi([q_1,\W d])}\sum_{\chi\text{ mod }q_1}^{\star}\vert \psi'(dx, \chi)\vert\sum_{q'\leq \frac{x^{\delta}}{[q_1,\W d]}}\frac{1}{\phi(q')}\\&\ll \log x\sum_{d\in \mathcal{D}}\dd\sum_{q_1\leq x^{\delta}}\frac{1}{\phi([q_1,\W d])}\sum_{\chi\text{ mod }q_1}^{\star}\vert \psi'(dx, \chi)\vert\\&\ll \log^2 x\sum_{d\in \mathcal{D}}\sum_{q_1\leq x^{\delta}}\frac{(q_1,\W d)}{\phi(q_1)}\frac{1}{d}\sum_{\chi\text{ mod }q_1}^{\star}\vert \psi'(dx, \chi)\vert,
\end{align*} 
where we have used the fact that $\phi(d)\gg d$ for $d\in \mathcal{D}$.
We divide the values of $d$ and $q_1$ into several (possibly overlapping) buckets. The first bucket is defined by the condition $q_1\leq \log^{2A}x$, and its contribution is bounded by $$\log^2 x\sum_{d\in\mathcal{D}}\sum_{q_1\leq \log^{2A}x}\frac{q_1}{\phi(q_1)}\frac{1}{d}\sum_{\chi\text{ mod }q_1}^{\star}\vert \psi'(dx, \chi)\vert\ll x\log^{2-2A}x\vert\mathcal{D}\vert$$ using Lemma \ref{bvugh}. The second bucket is defined by $q_1\leq e^{c\sqrt{\log x}}$ and $s\nmid q_1$, and its contribution is bounded by $$\log^2 x\sum_{d\in\mathcal{D}}\sum_{\substack{q_1\leq e^{c\sqrt{\log x}}\\ s\nmid q_1}}\frac{q_1}{\phi(q_1)}\frac{1}{d}\sum_{\chi\text{ mod }q_1}^{\star}\vert \psi'(dx, \chi)\vert\leq \log^2 x\frac{x}{e^{c\sqrt{\log x}}}\vert \mathcal{D}\vert$$ by Lemma \ref{bv2}. The third bucket is defined by $(q_1,2d)\leq\frac{2q_1}{\log^{2A}x}$, and its contribution is bounded by $$\log^{2-2A} x\sum_{d\in\mathcal{D}}\sum_{q_1\leq x^{\delta}}\frac{q_1}{\phi(q_1)}\frac{1}{d}\sum_{\chi\text{ mod }q_1}^{\star}\vert \psi'(dx, \chi)\vert\ll x\log^{6-2A}x\vert\mathcal{D}\vert$$ using Lemma \ref{bv1}. Note that by our hypothesis regarding $(s(\lfloor tx\rfloor),N)$, the pair $d$ and $q_1$ will be contained in this bucket if $\log^{2A}x<q_1\leq e^{c\sqrt{\log x}}$ and $s\mid q_1$. 
Thus, the remaining contribution is $$E:=\log^2 x\sum_{d\in \mathcal{D}}\sum_{\substack{e^{c\sqrt{\log x}}\leq q_1\leq x^{\delta}\\ (q_1,\W d)>\frac{2q_1}{\log^{2A}x}}}\frac{(q_1,\W d)}{\phi(q_1)}\frac{1}{d}\sum_{\chi\text{ mod }q_1}^{\star}\vert \psi'(dx, \chi)\vert.$$ We can bound this by $$\log^2 x\sum_{e^{c\sqrt{\log x}}\leq q_1\leq x^{\delta}}\frac{q_1}{\phi(q_1)}\sum_{\chi\text{ mod }q_1}^{\star}\max_{x\leq x'\leq tx}\frac{x}{x'}\vert \psi'(x', \chi)\vert D(q_1)$$ where $$D(q_1):=\#\Big\{d\in\mathcal{D}: (q_1,\W d)>\frac{2q_1}{\log^{2A}x}\Big\}.$$ Now, $$D(q_1)\leq \#\Big\{d'\in \mathcal{D}: d'\mid q_1, d'>\frac{2q_1}{\WW \log^{2A}x}\Big\}\max_{\substack{d'\in\mathcal{D}\\ d'>\frac{2q_1}{\WW \log^{2A}x}}}\#\big\{d\in \mathcal{D}: d'\mid d\big\}.$$ We have \begin{align*}\max_{\substack{d'\in\mathcal{D}\\ d'>\frac{2q_1}{\WW \log^{2A}x}}}\#\big\{d\in \mathcal{D}: d'\mid d\big\}&=\max_{\substack{d'\in\mathcal{D}\\ d'>\frac{2q_1}{\WW \log^{2A}x}}}\binom{\omega(N)-\omega(d')}{w-\omega(d')}\\&<\binom{\omega(N)}{w}\max_{\substack{d'\in\mathcal{D}\\ d'>\frac{2q_1}{\WW \log^{2A}x}}}2^{-\omega(d')},\end{align*} using the fact that $w<\omega(N)/2$.
If $d'>\frac{2q_1}{\WW \log^{2A}x}$ then $$\omega(d')>\frac{1}{\gamma\log\log x}\log\Big(\frac{2e^{c\sqrt{\log x}}}{\WW \log^{2A}x}\Big)>\log_2(e^{\sqrt[3]{\log x}})$$ for sufficiently large $x$. This means that $$\max_{\substack{d'\in\mathcal{D}\\ d'>\frac{2q_1}{\WW \log^{2A}x}}}\#\big\{d\in \mathcal{D}: d'\mid d\big\}<\frac{\vert\mathcal{D}\vert}{e^{\sqrt[3]{\log x}}}.$$ Also, $$\#\Big\{d'\in \mathcal{D}: d'\mid q_1, d'>\frac{2q_1}{\WW \log^{2A}x}\Big\}=\#\Big\{d'\mid (q_1,N), d'>\frac{2q_1}{\WW \log^{2A}x}\Big\}.$$ 
Considering the asymptotic behavior of primorials, we see that $(q_1,N)$ has at most $\log x$ factors, of which $d'$ can have lost at most $\log\log x$. 
Thus 
$$\#\Big\{d'\in \mathcal{D}: d'\mid q_1, d'>\frac{2q_1}{\WW \log^{2A}x}\Big\}<(\log x)^{\log\log x} = e^{(\log\log x)^2}.$$ 
We conclude that $$E\leq \log^2 x\frac{\vert\mathcal{D}\vert}{e^{\sqrt[3]{\log x}}}e^{(\log\log x)^2}\sum_{q_1\leq x^{\delta}}\frac{q_1}{\phi(q_1)}\sum_{\chi\text{ mod }q_1}^{\star}\max_{x\leq x'\leq tx}\frac{x}{x'}\vert \psi'(x', \chi)\vert.$$ Now, if $2^m$ is the smallest power of 2 greater than $t$ then $$\sum_{q_1\leq x^{\delta}}\frac{q_1}{\phi(q_1)}\sum_{\chi\text{ mod }q_1}^{\star}\max_{x\leq x'\leq tx}\frac{x}{x'}\vert \psi'(x', \chi)\vert$$ is asymptotically bounded by $$\sum_{i=1}^m\frac{1}{2^i}\sum_{q_1\leq x^{\delta}}\frac{q_1}{\phi(q_1)}\sum_{\chi\text{ mod }q_1}^{\star}\max_{2^{i-1}x\leq x'\leq 2^ix}\vert \psi'(x', \chi)\vert\ll x\log^5 x,$$ using Lemma \ref{bv1}. Putting everything together $$\sum_{d\in \mathcal{D}}\dd\sum_{\substack{q\leq x^{\delta}\\ \W d\mid q}}\max_{(a,q)=1}\Big\vert\psi(dx; q,a)-\frac{dx}{\phi(q)}\Big\vert\ll x\vert \mathcal{D}\vert\log^{6-2A}x.$$ Therefore, $$\sum_{\substack{q\leq x^{\delta}\\ \W d\mid q}}\max_{(a,q)=1}\Big\vert\psi(dx; q,a)-\frac{dx}{\phi(q)}\Big\vert\leq \frac{x}{\log^Ax}$$ for almost all $d\in\mathcal{D}$ if $A\geq 7$ (from which it follows for all $A$). Upon applying summation by parts to convert from $\psi$ to $\pi$, the stated result follows.
\end{proof}
Goldston, Pintz, and Y{\i}ld{\i}r{\i}m \cite{GPY:2010} prove a similar result which would imply that the desired inequality holds for all $d\in\mathcal{D}$. However, their result only holds for a sequence of values of $x$ which tend towards infinity, and so we use this averaged version for more regularity. For readers familiar with the AGP method, we will mention that this proposition will be used in lieu of Theorem 2.1 from \cite{AGP:1994}. Our result establishes stronger equidistribution, while their result lets the moduli come up to a larger power of $dx$ (5/12 instead of 1/6).
\section{Constructing $L$}\label{set-up}
We take $y$ to be a large parameter. The majority of this section will be devoted to proving the following proposition.
\begin{proposition}\label{set-up-prop}
There is a set $\mathcal{Q}\in \mathbb{P}\cap [1,y]$ and a value $\theta\geq \frac{1}{6}-\iota$ (which may depend on $y$) with the following properties:

\begin{enumerate}
\item $\vert \mathcal{Q}\vert\geq y^{\theta+\rho}\log^{-\kappa} y$
\item $m_0$ is a quadratic residue mod $q$ for each $q\in\mathcal{Q}$
\item $\frac{q-1}{2}$ is odd, square-free, and composite for each $q\in\mathcal{Q}$
\item There are $O(y^{\theta})$ primes $p$ such that $p\mid \frac{q-1}{2}$ for some $q\in\mathcal{Q}$
\item If $p\mid \frac{q-1}{2}$ with $q\in\mathcal{Q}$ then $p\geq y^{\theta(1+\iota)}$
\item For every prime $p$, $\#\{q\in \mathcal{Q}: p\mid \frac{q-1}{2}\}\leq y^{\rho}$
\item For every product $n$ of at most $y^{\rho}$ elements of $\mathcal{Q}$, for every non-principal character $\chi$ mod $n$, and for every real number $\beta$, there exist at least $\frac{y^\theta}{y^{2\iota}}$ elements $q\in\mathcal{Q}$ such that if $\beta_q$ is a real number with $\chi(q)=e(\beta_q)$ then $\vert \beta_q-\beta\vert\geq \frac{1}{2y^{\rho+\iota}}$
\end{enumerate}
\end{proposition}

Taken together, Properties 1, 4, and 6 give a good qualitative description of the group $(\mathbb{Z}/(\prod_{q\in\mathcal{Q}}q)\mathbb{Z})^{\times}$. The average number of copies of $\mathbb{Z}/p\mathbb{Z}$ (assuming there is at least one) is only slightly less than $y^{\rho}$, which is the maximal number of copies. This makes the group fairly even, in some sense. Clearly, the exponent of the group is considerably smaller than the order, as we want. Property 2 is important because it means that we won't have to worry about the 2-torsion of this group, which is the exception to the evenness. Property 3 is not particularly necessary but a simplifying assumption. Property 5 is strangely important; it ensures that each odd index subgroup has a fairly large index, meaning that a random element is quite unlikely to be in the subgroup. Finally, Property 7 is a weak equidistribution condition which is very important in ensuring that we can convert heuristic arguments into rigorous probabilistic ones.

To prove Proposition \ref{set-up-prop}, we will describe a probabilistic algorithm for generating $\mathcal{Q}$. If we can show that each property is satisfied with probability 99\% then there is a positive probability that all properties are satisfied. In fact, our method will show that each property is satisfied almost all the time. To be more precise, we say that an outcome is \emph{almost certain} if the probability that it doesn't occur is less than $e^{-y^\epsilon}$ for some $\epsilon>0$ which does not depend on $y$. 
We may further specify that the outcome is \emph{$\epsilon$-certain}, for such a value of $\epsilon$.
We also say that an outcome is \emph{very likely} if the probability that it doesn't occur is less than $e^{-(\log y)^{\epsilon}}$ for some $\epsilon>0$ which does not depend on $y$. Our construction will guarantee Properties 2, 3, and 5, and we will show that our set of primes satisfies Properties 1 and 4 almost certainly and very likely satisfies Properties 6 and 7. Everything in this section is in some sense overwhelmingly true; there is no subtlety, only technicalities to take care of.

Before beginning the proof of the proposition, we make the following combinatorial remark.  Suppose one chooses a number of objects each of which has an independent and identical probability of having a certain property.  If $y^\epsilon$ is less than the expected number of selected objects with the property,
then the number of  selected objects which actually possess the property will be at least $y^\epsilon/2$ with $\epsilon-\iota$-certainty.

\begin{proof}[Proof of Proposition \ref{set-up-prop}]
We begin with a lemma.
\begin{lemma}\label{seed}
When $y$ is sufficiently large,
there exist $\gg\frac{y}{\log^3y}$ primes $q\leq y$ such that $\frac{q-1}{2}$ is an odd square-free number with two or more prime factors, each of size at least $y^{1/6-\iota}$, and also such that $m_0$ is a quadratic residue mod $q$.
\end{lemma}
\begin{proof}
Note that $\big(\frac{p}{4p-1}\big)=1$ for each $p\mid m_0$, by quadratic reciprocity.
If $q\equiv -1\lessspacepmod{4m_0}$ then because each factor of $m_0$ is a quadratic residue mod $q$, it follows that $m_0$ is a quadratic residue mod $q$. 
For each $p\in [\frac{y^{1/6}}{2},y^{1/6}]$, we let $\bar p$ denote the multiplicative inverse of $p$ mod $4m_0$ in the interval $(0,4m_0)$.
Then, we try to apply Lemma \ref{rosser} with $b=-\bar p$, $r=4m_0$, $x=\frac{y}{2p}$, $d=2p$, $\delta=\frac{1}{5}$, and $\epsilon=\frac{1}{5}(1-\iota)$. The hypothesis of the lemma may not be satisfied for all $p$, but it will be satisfied on average since the average of the left-hand side of \eqref{bveq} is 
$$\sum_{p\in [\frac{y^{1/6}}{2},y^{1/6}]}\text{ }\sum_{\substack{q\leq (\frac{y}{2p})^{3/5-2\iota^2/5}\\ 2p\mid q}}\max_{(a,q)=1}\Big\vert\pi(y; q,a)-\frac{\pi(y)}{\phi(q)}\Big\vert$$
which can be asymptotically bounded by 
$$\sum_{q\leq y^{\frac{3-2\iota^2}{6}}}\max_{(a,q)=1}\Big\vert\pi(y; q,a)-\frac{\pi(y)}{\phi(q)}\Big\vert\ll \frac{y}{\log^A y}$$
for all fixed $A$ by the Bombieri-Vinogradov theorem. In this way, we obtain many primes $q$ such that $\frac{q-1}{2}$ has a prime factor on the order of $y^{1/6}$. We can additionally ensure that $\frac{q-1}{2}$ is not prime by discarding values of $q$ less than $2y^{1/6}$.
\end{proof}

It follows by dyadic decomposition that there exist values 
$$\frac{1}{6}-\iota\leq \theta_1\leq \theta_2\leq \cdots\leq \theta_{\ell}$$ 
such that there are at least $\frac{y}{\log^{10}y}$ primes $q\leq y$ of the form $2\prod_{i=1}^{\ell}p_i+1$, where each $p_i$ is a prime between $y^{\theta_i}$ and $2y^{\theta_i}$, and additionally such that $m_0$ is a quadratic residue mod $q$. Call this set of primes $\mathcal{Q}^{(1)}$. Recall that $\theta=\frac{1}{6}-2\iota$.

The first step of our random algorithm is to choose random subsets of primes $\mathcal{P}_i$. We give each prime between $y^{\theta_i}$ and $2y^{\theta_i}$ an independent probability of $\frac{y^{\theta}}{\#\{p\sim y^{\theta_i}\}}$ of being in $\mathcal{P}_i$. We let $\mathcal{Q}^{(2)}$ be the subset of $\mathcal{Q}^{(1)}$ consisting of elements of the form $2\prod_{i=1}^{\ell}p_i+1$ where for each $i$, $p_i\in\mathcal{P}_i$. We then independently give each element of $\mathcal{Q}^{(2)}$ a probability of $\frac{2y^{\theta+\rho}\log^{-\kappa} y}{\vert \mathcal{Q}^{(2)}\vert}$ of being in $\mathcal{Q}^{(3)}$. Of course, for this to make sense, $\vert \mathcal{Q}^{(2)}\vert$ has to be at least $2y^{\theta+\rho}\log^{-\kappa} y$, which we will show is almost certain to happen. Once this is established, it becomes clear that $\mathcal{Q}^{(3)}$ is almost certain to have Property 1. This is the first property we will show $\mathcal{Q}^{(3)}$ is almost certain to satisfy.

Fix a subset of $\mathcal{Q}^{(1)}$ that we will call $\mathcal{Q}^{\star}$. 
For any $\mathcal{P}_1,\ldots,\mathcal{P}_j$, we say that the collection of sets $\{\mathcal{P}_1,\ldots,\mathcal{P}_j\}$ \emph{performs generically} if there are at least $\prod_{i=1}^j\frac{\vert\mathcal{P}_i\vert}{y^{\theta_i}}\vert\mathcal{Q}^\star\vert$ primes in $\mathcal{Q}^{\star}$ which can be written in the form $q=2\prod_{i=1}^{\ell}p_i+1$ with $p_i\in \mathcal{P}_i$ for $i\leq j$ and $p_i\sim y^{\theta_i}$ for $i>j$. We write the number of such primes as $\pi(\mathcal{P}_1,\ldots,\mathcal{P}_j)$. Suppose the subcollection
$\{\mathcal{P}_1,\ldots,\mathcal{P}_j\}$ of the $\mathcal{P}_i$ constructed in the previous paragraph
performs generically, with $j<\ell$. Note that $$\pi(\mathcal{P}_1,\ldots,\mathcal{P}_j)\le\sum\limits_{\substack{p_i\in\mathcal{P}_i, 1\leq i\leq j\\ p_{j+1}\sim y^{\theta_{j+1}}}}\pi(\{p_1\},\ldots,\{p_j\},\{p_{j+1}\})$$ and also 
$$\pi(\{p_1\},\ldots,\{p_j\},\{p_{j+1}\})\leq 
\prod_{i=j+2}^\ell y^{\theta_i}\ll
y\prod_{i=1}^{j+1}y^{-\theta_i}.$$ Thus, the fraction of $(p_1,\ldots,p_j,p_{j+1})$ which perform generically is asymptotically at least $\frac{\vert\mathcal{Q}^\star\vert}{y}$. It is therefore almost certain that $\{\mathcal{P}_1,\ldots,\mathcal{P}_j,\mathcal{P}_{j+1}\}$ performs generically. By finite induction, it is almost certain that $\{\mathcal{P}_1,\ldots,\mathcal{P}_{\ell}\}$ performs generically. In particular, taking $\mathcal{Q}^{\star}=\mathcal{Q}^{(1)}$, we see that $\mathcal{Q}^{(2)}$ has Property 1 almost certainly, which means that $\mathcal{Q}^{(3)}$ also has this property almost certainly. Note that Property 4 is also almost certainly satisfied by $\mathcal{Q}^{(3)}$.

Next, we turn to Property 7. Let $s$ be a positive integer of size at most $y^{\rho}$. Let $Y=y^s$.

\begin{lemma}\label{largesieve}
For $N\in \mathbb{Z}_{>0}$ and complex numbers $\{a_m\}_{m\in\mathbb{Z}_{\geq 0}}$, $$\sum_{n\leq N}\frac{n}{\phi(n)}\sum_{\chi\emph{ mod }n}^{\star}\Big\vert\sum_{m=1}^Ma_m\chi(m)\Big\vert^2\ll (N^2+M)\sum_{m=1}^M\vert a_m\vert^2.$$
\end{lemma}
\begin{proof}
This is a version of the large sieve inequality. See, for instance, Theorem 2 of \cite{BD:1968}.
\end{proof}
In particular, \begin{equation}\label{large-sieve}\sum_{n\leq Y}\sum_{\chi\text{ mod }n}^{\star}\Big\vert \sum_{m\in \mathcal{S}}\chi(m)\Big\vert^2\ll Y^2\vert \mathcal{S}\vert\end{equation} where $\mathcal{S}$ is the set of integers with $2s$ prime factors, each of which lies in $\mathcal{Q}^{(1)}$.

We say that $n\leq Y$ is \emph{typical} if for every primitive character $\chi$ mod $n$, $$\Big\vert \sum_{m\in \mathcal{S}}\chi(m)\Big\vert=o(\vert \mathcal{S}\vert).$$ By the large sieve inequality (equation \eqref{large-sieve}), there are $O(\frac{Y^2}{\vert \mathcal{S}\vert})$ values of $n\leq Y$ which are atypical. We have $$\vert \mathcal{S}\vert> \Big(\frac{\vert \mathcal{Q}^{(1)}\vert}{2s}\Big)^{2s}\geq \Big(\frac{y}{2y^{\rho}\log^{10}y}\Big)^{2s}.$$ Therefore, there are at most $y^{2s(\rho+o(1))}$ atypical values of $n\leq Y$. 

Recall that $e(x)$ is a Lipschitz function. Consequently, there exist positive constants $c_1$ and $c_2$ such that when $\chi$ is a primitive character whose modulus is typical, the fraction of elements of $\mathcal{S}$ with the property that $\vert\arg^{\star}\chi(m)\vert\geq c_1$ is at least $c_2$, where $\arg^{\star}$ is a \emph{shifted argument function}, measuring the angle (in the range $(-\pi,\pi]$) with any particular vector on the unit circle.

Let $\chi$ be a primitive character mod some $n$, and let $\beta$ be an arbitrary real number. For an integer $m$, we define $\beta_m$ to be the value in the range $(\beta-1/2,\beta+1/2]$ for which $\chi(m)=e(\beta_m)$. We let $\mathcal{Q}^{\chi}_\beta$ be the set of elements of $\mathcal{Q}^{(1)}$ with $\vert \beta_q-\beta\vert\geq \frac{1}{y^{\rho+\iota}}$. If $\mathcal{Q}^{\chi}_\beta$ is small then the values of $\chi$ on a large subset of $\mathcal{Q}^{(1)}$ are tightly clustered. This in turn causes the character values on $\mathcal{S}$ to cluster around a single value. (Here it is important that each element of $\mathcal{S}$ has the same number of prime factors.)

We will now assume that $n$ is typical. We claim that the number of primes in $\mathcal{Q}^{(1)}$ which are also in $\mathcal{Q}^{\chi}_\beta$ is at least $\frac{y}{y^{\rho+\iota}}$. Indeed, if this is not the case then almost none of the elements of $\mathcal{S}$ are divisible by such a prime. For those $m\in \mathcal{S}$ not divisible, we have $$\vert \beta_m-2s\beta\vert\leq \sum_{q\mid m}\nu_q(m)\vert \beta_q-\beta\vert\leq \frac{2}{y^{\iota}}$$ where $\nu_q$ is the $q$-adic valuation. In particular, $\vert 2\pi(\beta_m-2s\beta)\vert<c_1$ for almost all $m\in\mathcal{S}$. This contradicts what we have established previously, since $2\pi(\beta_m-2s\beta)$ is the value at $\chi(m)$ of a shifted argument function. Therefore, $\mathcal{Q}^{\chi}_\beta$ will be large when $n$ is typical.

If we take $\mathcal{T}$ to be the set of $n\leq Y$ which have $s$ distinct prime factors, each of which lies in $\mathcal{Q}^{(3)}$, then we see that $$\vert \mathcal{T}\vert\leq \binom{\lfloor y^{\theta+\rho}\rfloor}{s}\leq \frac{y^{s(\theta+\rho+o(1))}}{s^s}$$ almost certainly. Moreover, there are almost certainly
$$\binom{\vert \mathcal{Q}^{(2)}\vert}{s}\geq\big(\frac{y^{\ell\theta}}{s\log^{O(1)}y}\big)^s\geq \frac{y^{s(\ell\theta-o(1))}}{s^s}$$ values which might be in $\mathcal{T}$, each of which is a priori equally likely to be a product of elements of $\mathcal{Q}^{(3)}$. Therefore, the probability that $\mathcal{T}$ contains any atypical values is bounded by $$\frac{\vert\mathcal{T}\vert\cdot\#\{\text{atypical integers in }[1,Y]\}}{\binom{\vert \mathcal{Q}^{(2)}\vert}{s}}\ll\frac{y^{2s(\rho+o(1))}}{y^{s(\ell\theta-\theta-\rho-o(1))}}=y^{s(3\rho-(\ell-1)\theta)+o(1)}.$$ It is therefore very likely that $\mathcal{T}$ doesn't contain any atypical values, an outcome which we will assume for the remainder of this paragraph. Take $\beta$ to be a real number and $\chi$ to be any primitive character mod an element of $\mathcal{T}$. Setting $\mathcal{Q}^{\star}$ equal to any $\mathcal{Q}^{\chi}_\beta$, we see from our earlier work that 
$\mathcal{Q}^{(2)}\cap \mathcal{Q}^{\chi}_\beta$ is large. In particular, the size of
$\mathcal{Q}^{(2)}\cap \mathcal{Q}^{\chi}_\beta$ almost certainly has at least $\frac{y^{\ell\theta}}{y^{\rho+2\iota}}$ elements.
Keeping in mind the combinatorial remark we made earlier, we may populate the following table:

\begin{center}
\begin{tabular}{|p{.2cm}|p{2.7cm}:p{2cm}|p{2.8cm}:p{2cm}|}
\hline
$i$&Expected lower bound for $\mathcal{Q}^{(i)}$&$\epsilon$-certainty&Expected bound for $\mathcal{Q}^{(i)}\cap \mathcal{Q}_\beta^\chi$&$\epsilon$-certainty\\
\hline
$1$&$y^{1-\iota}$&$\infty$&$y^{1-\rho-\iota}$&$\infty$\\
$2$&$y^{\ell\theta-\iota}$&$\theta_1-\iota$&$y^{\ell\theta-\rho-2\iota}$&$\theta_1-\rho-2\iota$\\
$3$&$y^{\theta+\rho-\iota}$&$\theta_1-\iota$&$y^{\theta-2\iota}$&$\theta_1-\rho-2\iota$\\
\hline
\end{tabular}
\end{center}
The key point here is that hardly any additional certainty is lost when passing from $\mathcal{Q}^{(2)}$ to $\mathcal{Q}^{(3)}$. This is because this stage of the culling is done uniformly, whereas
the selection of $\mathcal{Q}^{(2)}$ from $\mathcal{Q}^{(1)}$ involves choosing a random subset of a smaller set, namely primes on the order of $y^{\theta_1}$.
Note that we have been unnecessarily generous with subtracting $\iota$ terms and that stronger bounds also hold.

We now divide $[0,1]$ into subintervals of length $\frac{1}{y^{\rho+\iota}}$ and take our values of $\beta$ to be the endpoints of these subintervals. This is sufficient to tell us that for any $\beta$ value, $\vert \beta_q-\beta\vert\geq \frac{1}{2y^{\rho+\iota}}$ for many values of $q$. In particular, applying this argument for every primitive $\chi$ whose modulus is in $\mathcal{T}$ and then repeating this process for every $s$ up to  $y^{\rho}$, we can conclude that $\mathcal{Q}^{(3)}$  satisfies Property 7 with $\theta-2\rho-2\iota$-certainty, recalling that $\theta < \theta_1$.

Finally, let $p$ be a prime less than $y$. We are interested in the probability that $$\#\big\{q\in \mathcal{Q}^{(3)}: p\mid \frac{q-1}{2}\big\}>y^{\rho}.$$ We note that $$\#\big\{q\in \mathcal{Q}^{(2)}: p\mid \frac{q-1}{2}\big\}\leq 2y^{(\ell-1)\theta}$$ almost certainly. Conditional on this being the case, the probability we are studying can be bounded by 
\begin{align*}
\sum_{j=\lceil y^\rho\rceil}^{\lfloor 2y^{(\ell-1)\theta}\rfloor}\binom{\lfloor 2y^{(\ell-1)\theta}\rfloor}j\Big(\frac{2y^{\theta+\rho}\log^{-\kappa} y}{\vert \mathcal{Q}^{(2)}\vert}\Big)^j
&<\sum_{j=\lceil y^\rho\rceil}^{\lfloor 2y^{(\ell-1)\theta}\rfloor}\Big(\frac{y^{\ell\theta+\rho}}j\log^{-\kappa} y\frac{\log^{11}y}{y^{\ell\theta}}\Big)^j\\
&<\sum_{j\ge y^\rho} (\log y)^{(11-\kappa)j}.
\end{align*} Since $\kappa$ is assumed to be large, this probability is vanishingly small. In particular, it is very likely that $\mathcal{Q}^{(3)}$ satisfies Property 6. Since it is very likely that $\mathcal{Q}^{(3)}$ has each property, there is a particular choice of $\mathcal{Q}^{(3)}$ which has each property.

\end{proof}

Let $\mathcal{Q}$ be a set which has the properties enumerated in the statement of Proposition \ref{set-up-prop}. There is another condition we will want our set of primes to satisfy before moving forward. We set $z=\lfloor y^{(1+\delta-3\iota)\lfloor y^{\theta+\iota\rho}\rfloor}\rfloor$. Let $t$ be a divisor of $s(z)$ with the property that $t$ is greater than $\log^{100} z$ but each proper divisor of $t$ is less than $\log^{100} z$. Then the number of distinct prime factors of $t$ is less than $\log\log z=O(\log y)$. We let $\mathcal{Q}'=\{q\in\mathcal{Q}: q\nmid t\}$ and say because $(s(z),\prod_{q\in\mathcal{Q}'}q)\leq \frac{s(z)}{\log^{100} z}$ that $\mathcal{Q}'$ satisfies Property 8. 

Now, we have the following combinatorial lemma.
\begin{lemma}
Let $\mathcal{A}$ be a set with $k$ elements and $n$ be a fixed integer. Let $\mathcal{B}$ be a subset of $\{(a_1,\ldots,a_n): a_i=a_j\implies i=j\}$ such that if $a\in \mathcal{A}$ then the proportion of elements of $\mathcal{B}$ with $a$ as a component is $o(1)$. Then for a random partition of $\mathcal{A}$ into $\mathcal{A}_1$ and $\mathcal{A}_2$ (where each element in $\mathcal{A}$ is independently assigned to $\mathcal{A}_1$ or $\mathcal{A}_2$ with equal likelihood), there is a fixed positive lower bound (dependent on $n$ but not $\mathcal{A}$ or $k$) on the probability that $$\vert \mathcal{A}_1^n\cap \mathcal{B}\vert\geq \frac{\vert \mathcal{B}\vert}{2^{2n+1}}$$ and $$\vert \mathcal{A}_2^n\cap \mathcal{B}\vert\geq \frac{\vert \mathcal{B}\vert}{2^{2n+1}}.$$
\end{lemma}
\begin{proof}
Pick two elements $b_1$ and $b_2$ in $\mathcal{B}$ with the property that no element of $\mathcal{A}$ is a component of both of them. Note that if $b_1$ and $b_2$ are chosen randomly, this happens with probability $1-o(1)$ (this becomes clear if you first choose $b_1$ and then consider the chance that $b_2$ shares any particular component). Let $\mathcal{A}_1$ and $\mathcal{A}_2$ randomly partition $\mathcal{A}$, and for $i\in\{1,2\}$, let $\mathcal{A}^n_i$ be the subset of $\prod_{j=1}^n\mathcal{A}$ given by $\prod_{j=1}^n\mathcal{A}_i$. The probability that $b_1\in \mathcal{A}_1^n$ and $b_2\in \mathcal{A}_2^n$ is $\frac{1}{2^{2n}}$. Thus, the expected value of the number of pairs $(b_1,b_2)$ with $b_1\in \mathcal{A}_1^n$ and $b_2\in \mathcal{A}_2^n$ is $\frac{\vert \mathcal{B}\vert^2(1-o(1))}{2^{2n}}$. The lemma follows by noting that the maximum value of this quantity is $\vert\mathcal{B}\vert^2$.
\end{proof}
Let $\mathcal{A}$ be the disjoint union over $i\in\{1,\ldots,\ell\}$ of $$\big\{p\sim y^{\theta_i}: \exists q\in\mathcal{Q}'\text{ with }p\mid \frac{q-1}{2}\big\}.$$ Let $$\mathcal{B}=\big\{(p_1,\ldots,p_{\ell}): p_i\in\mathcal{P}_i\forall i, 2\prod_{i=1}^{\ell}p_i+1\in\mathcal{Q}'\big\}.$$ (Note that this set may be slightly larger in size than $\mathcal{Q}'$ if $\theta_i\approx \theta_{i+1}$ for some $i$.) We randomly partition $\mathcal{A}$ into $\mathcal{A}_1$ and $\mathcal{A}_2$. For $i\in\{1,2\}$, we let $\mathcal{Q}_i$ be the set of elements of $\mathcal{Q}'$ which correspond to elements of $\mathcal{A}_i^n\cap \mathcal{B}$. Applying our lemma, there is a positive probability that $\vert \mathcal{Q}_1\vert$ and $\vert \mathcal{Q}_2\vert$ are both greater than $y^{\theta+\rho}\log^{-\kappa-1} y$. Moreover, it is almost certain that the following statement holds:
\vskip 5pt
\noindent \textbf{Property 7*}. 
For each $i\in\{1,2\}$, for every product $n$ of at most $y^{\rho}$ elements of $\mathcal{Q}_i$, for every non-principal character $\chi$ mod $n$, and for every real number $\beta$, there exist at least $y^{\theta-3\iota}$ elements $q\in\mathcal{Q}_{3-i}$ such that if $\chi(q)=e(\beta_q)$ then $\vert\beta_q-\beta\vert\geq \frac{1}{2y^{\rho+\iota}}$. 
\vskip 5pt
Therefore, we can construct sets $\mathcal{Q}_1$ and $\mathcal{Q}_2$ which satisfy Property 7*, which satisfy Properties 2--6 verbatim (as well as Property 8), and which satisfy a weaker version of Property 1 (containing at least $y^{\theta+\rho}\log^{-\kappa-1} y$ elements). 

Take $L_i=\prod_{q\in\mathcal{Q}_i}q$ and $$\mathcal{D}_i=\{d\mid L_i: \omega(d)=w\},$$ where $w=\lfloor y^{\theta+\iota\rho}\rfloor$. Also take $$L_i(p)=\prod\limits_{\substack{q\in\mathcal{Q}_i\\ q\equiv 1\lessspacepmod{p}}}q.$$ Finally, set $L=L_1L_2$. The division into two sets can be explained as follows. For each prime $q\mid L$, we will want to obtain a good upper bound for the probability that $dk+1$ is in a proper subgroup of odd index of $(\mathbb{Z}/q\mathbb{Z})^{\times}$, for a given method of choosing $d\mid L$. The natural approach would be to choose $d$ uniformly from the divisors of $L$ of a certain size. In order to then obtain an upper bound, we would need the divisors to equidistribute mod each $q\mid L$. However, unless the divisors are extremely small, there will be a clear bias towards the residue 0 mod $q$. With two sets of primes, we can make sure that the divisors of $L$ which are supposed to equidistribute mod elements of $\mathcal{Q}_1$ have prime factors drawn from $\mathcal{Q}_2$ and vice versa.

\section{Constructing Families of Primes}\label{k-constr}

Fix $i\in \{1,2\}$. Let $x=y^{\frac{w}{\delta-3\iota}}$, where again $w=\lfloor y^{\theta+\iota\rho}\rfloor$.  We apply Proposition \ref{bv3} with these values of $x$ and $w$, with $N=L_i$, and with $A=10$, letting $\mathcal{D}'_i$ be the resulting subset of $\mathcal{D}_i$ with the desired property. Let $d^{\star}_i\lessspacepmod{q_0}$ be the most common residue among the elements of $\mathcal{D}'_i$. 
Let $R_1=\log x/(\log\log x)^3$ and $R_2=\log x\log\log x$. Recall that $T$ has been taken to be a large integer. Let $\mathcal{R}$ consist of primes in the range $[R_1,R_2]$, together with $2^T$ and all odd primes that divide $q-1$ for any $q\mid L$.

Fix $d\in\mathcal{D}'_i$. If $p\in [x^{\iota},x^{1-\iota}]$ then by Lemma \ref{Selberg}, $$\#\{k'\leq \frac{x}{\W p}: \W dk'p+1\in\mathbb{P}, \text{P}^-(k')\geq x^{\iota}\}\ll \frac{x}{p\log^2 x}.$$ 
If $r\in\mathcal{R}$ then the number of $k\leq x$ such that $dk+1$ is prime, $k$ is even, the smallest prime factor of $k/2$ is greater than $x^{\iota}$ but not equal to $k/2$ itself, and $k$ has a prime factor $p\equiv 1\lessspacepmod{r}$ is bounded asymptotically by 
$$\sum_{\substack{p\in [x^{\iota},x^{1-\iota}]\\ p\equiv 1\lessspacepmod{r}}}\frac{x}{p\log^2 x}\ll \frac{x}{r\log^2 x},$$ where we have used an upper bound on the number of primes which are 1 mod $r$ which can be provided by, for instance, the Siegel-Walfisz theorem. Let $\mathcal{K}'_d$ be the set of $k$ which satisfy these properties for any $r\in\mathcal{R}$. Then the size of $\mathcal{K}'_d$ is bounded asymptotically by 
\begin{align*}\frac{x}{\log^2 x}\sum_{r\in \mathcal{R}}\frac{1}{r}&=\frac{x}{\log^2 x}\bigg(\frac{1}{2^T}+O\Bigl(\log\Big(\frac{\log\log x+\log \log\log x}{\log\log x-3\log \log\log x}\Big)+y^{-\theta\iota}\Bigr)\bigg)\\
&= \frac{x}{\log^2 x}\bigg(\frac{1}{2^T}+O\Bigl(\frac{\log\log\log x}{\log\log x}\Bigr)\bigg).
\end{align*} 
Now, let $\mathcal{K}$ be the set of positive integers $k\leq x$ with the following properties:
\begin{itemize}
\item $k$ is even
\item $\text{P}^-(k/2)>x^{\iota}$
\item $k/2$ is not prime
\item $d^{\star}_ik\equiv a_0-1\lessspacepmod{q_0}$
\item $(\mathbb{Z}/k\mathbb{Z})^{\times}$ has no subgroup whose index is a prime in the interval $[R_1,R_2]$
\item There are no order $2^T$ characters mod $k$
\item $(\phi(k),\phi(L))$ is a power of 2
\end{itemize}
As we have just shown that $\sum_{r\in \mathcal{R}}\frac 1r$ is bounded, $|\mathcal{K}| \gg \frac x{\log x}$.
Note also that if $k\in\mathcal{K}$ then $(k,L)=1$. We want to use Lemma \ref{rosser} to show that 
\begin{equation}
\label{lemma-condition}
\{k\in\mathcal{K}: dk+1\in\mathbb{P}\}\gg \frac{x}{\log^2 x}.
\end{equation}
Instead of directly applying the lemma with our values of $d$ and $x$, we will take $\tilde d := 2pd$ and $\tilde x := \frac{x}{2p}$, which depend on the prime $p$, which varies in the range $(x^{\iota},2x^{\iota})$. This is a way of ensuring that $k/2$ is not prime, just like in Lemma \ref{seed}. In the application, we also take $\delta=\frac{1}{6}-3\iota$, $\epsilon=\iota$, $r=q_0$, and $b=\overline{d_i^{\star}}(a_0-1)$, where $\overline{d_i^{\star}}$ denotes a representative of
the multiplicative inverse of $d_i^\star$ modulo $q_0$. The hypothesis that needs to be established is that 
$$\sum_{\substack{q\leq \tilde x^{1/6-3\iota+2(1+\iota)\iota}\\ \tilde d\mid q}}\max_{(a,q)=1}\Big\vert \pi(dx; q,a)-\frac{\pi(dx)}{\phi(q)}\Big\vert\leq \frac{\tilde x}{\log^5 \tilde x}$$ 
on average, as $p$ varies, for a specific $d$ for which \eqref{lemma-condition} is being established. The left-hand side, summed over all primes $p$ in our range, can be asymptotically bounded by 
$$\sum_{\substack{q\leq x^{1/6-3\iota+2(1+\iota)\iota}\\ 2d\vert q}}\max_{(a,q)=1}\Big\vert \pi(dx; q,a)-\frac{\pi(dx)}{\phi(q)}\Big\vert.$$ By definition of $\mathcal{D}'_i$, therefore, the desired bound holds on average. To finish up the third condition, we again need to subtract off the small values of $k$ (those less than $4x^{\iota}$). Finally, the last three conditions on $k$ are taken care of by subtracting off $\mathcal{K}'_d$.

Applying this to every $d\in \mathcal{D}'_i$, we find $\gg \frac{x}{\log^2x}\vert \mathcal{D}_i\vert$ pairs $(d,k)$ with $d\in\mathcal{D}_i$ congruent to $d^{\star}_i$ mod $q_0$, $k\in\mathcal{K}$, and $dk+1\in\mathbb{P}$. Since $\vert\mathcal{K}\vert\ll \frac{x}{\log x}$,
there are, in fact, $\gg \frac{|\mathcal{K}||\mathcal{D}_i|}{\log x}$ such pairs.
Let $\mathcal{K}_i$ be the set of $k\in\mathcal{K}$ for which there are at least $\frac{\vert\mathcal{D}_i\vert}{\log x\log\log x}$ values of $d\in\mathcal{D}_i$ with $d\equiv d^{\star}_i\lessspacepmod{q_0}$ such that $dk+1$ is a prime. Let $\alpha_i=\frac{\vert\mathcal{K}\vert}{\vert\mathcal{K}_i\vert}$. For $k\in\mathcal{K}$, let 
$$\mathcal{D}_i^{(k)}:=\{d\in\mathcal{D}_i, d\equiv d^{\star}_i\ (\textrm{mod}\ q_0): dk+1\in\mathbb{P}\}.$$ 
We know the average size of $\mathcal{D}_i^{(k)}$ is $\gg \frac{|\mathcal{D}_i|}{\log x}$.
Since only those $k$ in $\mathcal{K}_i$ contribute substantially, the average size of $\mathcal{D}_i^{(k)}$ as $k$ ranges over $\mathcal{K}_i$ is
$\gg \frac{\alpha_i\vert\mathcal{D}_i\vert}{\log x}$.

Pick distinct elements $d,d'\in\mathcal{D}_i$, congruent to $d^{\star}_i$ mod $q_0$. 
Then the probability that $dk+1$ and $d'k+1$ are both primes when $k$ is chosen randomly from $\mathcal{K}_i$ is 
$\gg \big(\frac{\alpha_i}{\log x}\big)^2$. (Indeed, the probability is minimized if each element of $\mathcal{K}_i$ is involved in the same number of primes.) In particular, the number of $k\in \mathcal{K}$ for which $dk+1$ and $d'k+1$ are primes congruent to $a_0$ mod $q_0$ is $\gg \vert\mathcal{K}_i\vert\big(\frac{\alpha_i}{\log x}\big)^2$. On the other hand, by the Selberg sieve, the number of such $k$ is $O\big(\frac{x}{\log^3 x}\big)$. (This is analogous to Lemma \ref{Selberg} but taking $\mathcal{A}$ to be the set of products $k(dk+1)(d'k+1)$; it follows from (7.143) and Theorem 7.14 of \cite{FI:2010}.) From this, we have $\alpha_i |\mathcal{K}| \ll \frac x{\log x}$, so 
$\alpha_i=O(1)$, i.e. 
$$\vert\mathcal{K}_i\vert\gg \vert\mathcal{K}\vert.$$

\section{Constructing Carmichael Numbers}\label{carmichaelconstruction}
In this section, we construct Carmichael numbers congruent to $\frac{1}{m_0}$ mod $L_2$ whose prime factors are of the form $dk_1+1$ (with $d\in \mathcal{D}_1$). These are not the Carmichael numbers we are ultimately interested in, but they will be important factors. Essentially, our goal for now is to deal with the modular constraints modulo $L_1$ and $L_2$; we will worry about the constraints modulo $k_1$ and $k_2$ in later sections. We start off with an equidistribution result.

\begin{lemma}\label{equi}
If $p$ is a prime with $L_2(p)>1$ then for each residue class $a\in (\mathbb{Z}/L_2(p)\mathbb{Z})^{\times}$, $$\#\{d\in \mathcal{D}_1: d\equiv a\lessspacepmod{L_2(p)}\}=\frac{\vert\mathcal{D}_1\vert}{\phi(L_2(p))}\big(1+o(1)\big).$$
\end{lemma}
\begin{proof}
Let $\epsilon=\frac{w}{\omega(L_1)}$, where as before, $w$ is the number of prime factors each element of $\mathcal{D}_1$ has. Using the definition of $w$, we have $\epsilon>y^{-(1-\iota)\rho}$. Our initial goal will be to establish a bound for 
$$f(a):=\int_0^{1/\epsilon}dt\sum_{\chi\text{ mod }L_2(p)}\overline{\chi(a)}\prod_{q\in \mathcal{Q}_1}\Big((1-\epsilon)e^{-2\pi it\epsilon}+\epsilon e^{2\pi it(1-\epsilon)}\chi(q)\Big),$$ 
which, up to a multiplicative constant, is equal to the number of elements of $\mathcal{D}_1$ which are congruent to $a$ mod $L_2(p)$.
Note that 
$$\Big\vert (1-\epsilon)e^{-2\pi it\epsilon}+\epsilon e^{2\pi it(1-\epsilon)}\chi(q)\Big\vert^2=1-2\epsilon(1-\epsilon)\big(1-\Re(e^{2\pi it}\chi(q))\big).$$ Let $\chi$ be a non-principal character mod $L_2(p)$. By Property 6 of Proposition \ref{set-up-prop}, $\omega(L_2(p))\leq y^{\rho}$. Thus Property 7* applies to $\chi$, from which we learn that,
for fixed $t$,
 there exist at least $\frac{y^{\theta}}{y^{3\iota}}$ elements $q\in\mathcal{Q}_1$ such that if $\chi(q)=e(\beta_q)$ then $\vert \beta_q+t\vert\geq \frac{1}{2y^{\rho+\iota}}$. 
 When this is the case, $$1-\Re(e^{2\pi it}\chi(q))\geq 1-\cos\big(\frac{\pi}{y^{\rho+\iota}}\big)>\frac{2}{y^{2(\rho+\iota)}}$$ so \begin{align*}\Big\vert (1-\epsilon)e^{-2\pi it\epsilon}+\epsilon e^{2\pi it(1-\epsilon)}\chi(q)\Big\vert&=\Big\vert 1-\epsilon(1-e^{2\pi i(t+\beta_q)})\Big\vert\\&<\Big\vert 1-\frac{2\epsilon}{y^{2(\rho+\iota)}}+\epsilon\frac{\pi i}{y^{\rho+\iota}}\Big\vert\\&=\sqrt{(1-\frac{2\epsilon}{y^{2(\rho+\iota)}})^2+\frac{\epsilon^2\pi^2}{y^{2(\rho+\iota)}}}<1-\frac{\epsilon}{y^{2(\rho+\iota)}}.\end{align*}
Therefore, for any $t$, 
$$\sum_{\substack{\chi\text{ mod }L_2(p)\\ \chi\neq\chi_0}}\prod_{q\in \mathcal{Q}_1}\Big((1-\epsilon)e^{-2\pi it\epsilon}+\epsilon e^{2\pi it(1-\epsilon)}\chi(q)\Big)$$ 
is less than 
\begin{equation}\label{comp1}\phi(L_2(p))\Big(1-\frac{\epsilon}{y^{2(\rho+\iota)}}\Big)^{\frac{y^{\theta}}{y^{3\iota}}}<y^{y^{\rho}}e^{-y^{\theta-2(\rho+\iota)-3\iota-(1-\iota)\rho}}.\end{equation} 
On the other hand, the $\chi_0$ term contributes
\begin{align*}
\int_0^{1/\epsilon}dt\Big((1-\epsilon)e^{-2\pi it\epsilon}+\epsilon e^{2\pi it(1-\epsilon)}\Big)^{\omega(L_1)}
&=\frac{1}{\epsilon}(1-\epsilon)^{\omega(L_1)-w}\epsilon^w\binom{\omega(L_1)}{w} \\
&>\frac 1{\epsilon\omega(L_1)} = \frac 1w.
\end{align*}
Comparing this with \eqref{comp1} (and remembering that $\theta= 4\rho+6\iota$), 
we conclude that $$f(a)=\frac{1}{\phi(L_2(p))}\int_0^{1/\epsilon}dt\Big((1-\epsilon)e^{-2\pi it\epsilon}+\epsilon e^{2\pi it(1-\epsilon)}\Big)^{\omega(L_1)}\big(1+o(1)\big).$$ A standard application of Fourier analysis and character theory tells us that $f(a)$ is equal to $$\frac{1}{\epsilon}\phi(L_2(p))(1-\epsilon)^{\omega(L_1)-w}\epsilon^w\#\{d\in\mathcal{D}_1: d\equiv a\lessspacepmod{L_2(p)}\}.$$ It follows that $$\#\{d\in\mathcal{D}_1: d\equiv a\lessspacepmod{L_2(p)}\}=\frac{\vert\mathcal{D}_1\vert}{\phi(L_2(p))}\big(1+o(1)\big).$$
\end{proof}
Next we make a simple but important observation.
\begin{lemma}\label{added}
If $p$ is a prime with $L_2(p)>1$ then the fraction of elements of $\mathbb{Z}/L_2(p)\mathbb{Z}$ which are not units is $o(1)$.
\end{lemma}
\begin{proof}
The proportion of non-units is no more than $$\sum_{q\mid L_2(p)}\frac{1}{q}\ll \frac{y^{\rho}}{y^{2\theta(1+\iota)}}=o(1).$$

\end{proof}
Let $s=\lceil y^{2\rho}\rceil$. Let $k$ be a fixed integer relatively prime to $L_2$. Fix an odd prime $p$ with $L_2(p)>1$, and fix a subgroup $G$ of $(\mathbb{Z}/L_2(p)\mathbb{Z})^{\times}$ of index $p$. Consider the probability that $$d_1k+1,\ldots,d_sk+1\in G$$ where each $d_i$ is chosen uniformly (and independently) from $\mathcal{D}_1$. This of course is equal to the $s$th power of the probability that $dk+1$ is contained in $G$ for a single $d$ chosen uniformly from $\mathcal{D}_1$, which by Lemmas \ref{equi} and \ref{added} is $\frac{1}{p}\big(1+o(1)\big)$. Thus the overall probability is bounded by $$\Big(\frac{1}{p}\big(1+o(1)\big)\Big)^s\leq y^{-s(\theta(1+\iota)-o(1))}.$$ The importance of isolating one prime at a time cannot be overstated. As mentioned previously, we cannot possibly ask for equidistribution modulo $L_2$. By construction of $L_2$, however, we can show that clustering in a subgroup implies poor equidistribution modulo an integer which is much smaller than $L_2$, closer in size to $y$ than to $x$.

Note that $(\mathbb{Z}/L_2(p)\mathbb{Z})^{\times}$ contains $\frac{p^{\omega(L_2(p))}-1}{p-1}$ subgroups of index $p$. Thus, the probability that any such $G$ exists for any odd prime $p$ is bounded by 
$$y^{-s(\theta(1+\iota)-o(1))}\cdot y^{y^{\rho}}\cdot y.$$ 
That is, the probability that $d_1k+1,\ldots,d_sk+1$ all lie in a proper subgroup of $(\mathbb{Z}/L_2\mathbb{Z})^{\times}$ of odd index is bounded by $y^{-s(\theta(1+\iota)-o(1))}$. 

Let $\mathcal{K}^{\star}_1$ be the set of $k\in\mathcal{K}_1$ for which there is no subgroup of odd index of $(\mathbb{Z}/L_2\mathbb{Z})^{\times}$ which contains $dk+1$ for at least 
$\frac{\vert\mathcal{D}_1\vert}{10\log x\log\log x}$ values of $d\in\mathcal{D}_k^{\1}$. Suppose that $$\#\{k\in \mathcal{K}_1, k\notin \mathcal{K}^{\star}_1\}\geq \frac{x}{\log^2 x}.$$ Then there would be at least 
$$\frac{x}{\log^2 x}\big(\frac{\vert\mathcal{D}_1\vert}{10\log x\log\log x}\big)^s$$ 
values of $k$ and $d_1,\ldots,d_s\in\mathcal{D}_1$ (where repeats are allowed and order matters) with all of $\{d_ik+1\}_i$ in some subgroup. However, we just showed that there are at most $\frac{x\vert\mathcal{D}_1\vert^s}{y^{s(\theta(1+\iota)-o(1))}}$ sets of values with that property. Remembering that $\log x=O(y^{\theta+\iota\rho}\log y)$, this gives a contradiction. Generically, therefore, our primes of the form $dk+1$ will be able to hit any quadratic residue mod $L_2$, including $\frac{1}{m_0}$.

For the remainder of the section, we will focus on making this statement more precise. For a value of $k\in \mathcal{K}^{\star}_1$ that we treat as fixed for now, let $\mathcal{P}^{\1}_k$ be the set of primes of the form $dk+1$, with $d\in\mathcal{D}_k^{\1}$.

\begin{lemma}\label{standard}
Let $G$ be a subgroup of $(\mathbb{Z}/N\mathbb{Z})^{\times}$. For any sequence of $n$ elements in $G$, where $n\in [2,N^{\lambda(N)}]$, there exist at least $\frac{2^n}{n^{\lambda(N)^2}}$ subsequences whose product is the identity.
\end{lemma}
\begin{proof}
We define the Davenport constant $D(H)$ of a group $H$ to be the smallest integer such that every sequence of elements in $H$ of length $D(H)$ has a non-empty subsequence which multiplies to the identity. (This is equal to $n(H)+1$, in the notation of \cite{AGP:1994}.) Certainly $D(G)\leq D((\mathbb{Z}/N\mathbb{Z})^{\times})$. By Theorem 1.1 of \cite{AGP:1994}, $D(G)\leq \lambda(N)^2$. Take a sequence of $n$ elements in $G$. By Proposition 1.2 of \cite{AGP:1994}, there exist at least $\binom{n}{\lfloor n/2\rfloor}\binom{n}{\lambda(N)^2}^{-1}> \frac{2^n}{n^{\lambda(N)^2}}$ subsequences whose product is the identity, assuming $\lfloor n/2\rfloor>\lambda(N)^2$. Otherwise, the statement is trivial, since by convention the empty subsequence multiples to the identity.

\end{proof}

Let $J=\lambda(L)^{10}$. We wish to choose distinct primes $p_1,\ldots,p_J$ from $\mathcal{P}^{\1}_k$. We want to ensure that for each particular odd order subgroup mod $L_2$, most of the primes are not in the subgroup. This is certainly possible, as a random choice suffices. Indeed, the probability for a given subgroup of this not being the case is 
1 over a number whose logarithm's logarithm is comparable to a power of $y$.

Let $P$ be the product of these primes. Let 
$$w_0:=\begin{cases}\frac{\ell_0+1}2&\text{ if $\ell_0$ is odd}\\
\frac{\ell_0+2}4&\text{ if $\ell_0\equiv 6\lessspacepmod 8$}\\
\frac{3\ell_0+2}4&\text{ if $\ell_0\equiv 2\lessspacepmod 8$,}
\end{cases}$$
where $\ell_0$ is as in Proposition~\ref{workhorse} and, in particular, not divisible by $4$.
Note that  $4w_0\equiv 2\lessspacepmod{\ell_0}$ and that the additive order of $w_0$ is odd.

\begin{proposition}\label{Lresidue}
For any quadratic residue $\square$ mod $L_2$, there exist at least $\frac{2^J}{\phi(L_2)J^{\lambda(L)^2}}$ divisors of $P$ with $w_0$ mod $\ell_0$ prime factors congruent to 1 mod $L_1$ and $\square$ mod $L_2$.
\end{proposition}
\begin{proof}
Let $g$ be the element $(1,\square,w_0)$ in 
$$G:=(\mathbb{Z}/L_1\mathbb{Z})^{\times}\oplus (\mathbb{Z}/L_2\mathbb{Z})^{\times}\oplus \mathbb{Z}/\ell_0\mathbb{Z}.$$ 
We decompose $G$ as $G_1\oplus G_2\oplus G_3\oplus H$ where $G_i$ is the maximal subgroup of odd order in $(\mathbb{Z}/L_i\mathbb{Z})^{\times}$ for $i\in\{1,2\}$ and the maximal subgroup of odd order in $\mathbb{Z}/\ell_0\mathbb{Z}$ for $i=3$ and $H$ is the 2-Sylow subgroup of $G$. 
We are interested in the expression 
$$\sum_{\chi_i\in \hat{G}_i}\sum_{\chi_4\in \hat{H}}\overline{\chi_1(g)\chi_2(g)\chi_3(g)\chi_4(g)}\prod_{j=1}^J\frac{1+\chi_1(p_j)\chi_2(p_j)\chi_3(p_j)\chi_4(p_j)}{2},$$ 
where $p_j$ is viewed as the element 
$$(p_j\lessspacepmod{L_1},p_j\lessspacepmod{L_2},1)$$
(according to the original description of $G$).
We can simplify this to $$\sum_{\chi_i\in \hat{G}_i}\sum_{\chi_4\in \hat{H}}\overline{\chi_2(g)\chi_3(g)}\prod_{j=1}^J\frac{1+\chi_1(p_j)\chi_2(p_j)\chi_3(p_j)\chi_4(p_j)}{2}$$ since $\chi_4(g)=1$ (since the order of $\chi_4$ is a power of 2 and the order of $g$ is odd).

If $$\chi_1(p_j)\chi_2(p_j)\chi_3(p_j)\chi_4(p_j)=1$$ then 
$$\chi_1(p_j)=\chi_2(p_j)=\chi_3(p_j)=\chi_4(p_j)=1.$$ 
We know that most of the $p_j$ lie outside the kernel of $\chi_2$, assuming $\chi_2$ is non-principal. 
The same is true, to a stronger degree, of $\chi_3$.
Therefore, the expression we are interested in is equal to $$\sum_{\chi_1\in \hat{G}_1}\sum_{\chi_4\in \hat{H}}\prod_{j=1}^J\frac{1+\chi_1(p_j)\chi_4(p_j)}{2}$$ up to an error term of size 
$$O\bigg(L\Big\vert \frac{1+e(\frac{1}{\lambda(L)\ell_0})}{2}\Big\vert^{J/2}\bigg)
=O\Big(Le^{-\lambda(L)^3}\Big).$$ 
Thus 
$$\frac{\vert G\vert}{2^J}\{\Pi\mid P, \Pi\equiv 1\lessspacepmod{L_1},\Pi\equiv \square\lessspacepmod{L_2},\omega(\Pi)\equiv w_0\lessspacepmod{\ell_0}\}$$ is equal to 
$$\frac{\vert G_1\vert \vert H\vert}{2^J}\{\Pi\mid P: \Pi=\text{Id}\text{ as an element of }G_1\oplus H\}+O(Le^{-\lambda(L)^3}).$$ We conclude from Lemma \ref{standard} that the number of divisors of $P$ with the desired properties is at least $\frac{1}{\phi(L_2)}\frac{2^J}{J^{\lambda(L)^2}}$.
\end{proof}
We apply this proposition with $\square=\frac{1}{m_0}$. The set of divisors of $P$ which have $w_0\lessspacepmod{\ell_0}$ prime factors and which are congruent to 1 mod $L_1$ and $\frac{1}{m_0}$ mod $L_2$ is large enough that it must contain divisors with $J(1-o(1))$ prime factors. We choose one such divisor for each possible choice of $P$ to form the set $\mathcal{A}^{\1}_{k}$. This being done, we can stop worrying about the distribution of our primes modulo $L$.

\section{Escaping Subgroups}\label{heart}
We now want to count the number of $k_1$ and $k_2$ in $\mathcal{K}$ such that almost all of $\mathcal{P}^{\1}_{k_1}$ lies in a proper subgroup mod $k_2$. Heuristically, we would expect that there would be relatively few such values, but this is not quite as clear as it might seem at first glance. If many values of $k_1$ and $k_2$ were both divisible by 3, for instance, there would clearly be a problem. From this example, we see the importance of the elements of $\mathcal{K}$ having no small factors (other than two).

In some sense, the result we are looking for is a generalization of a theorem of Goldfeld \cite{G:1968} about Artin's conjecture on average. Roughly speaking, Goldfeld shows that if $a$ is a generic integer then it generates $(\mathbb{Z}/p\mathbb{Z})^{\times}$ for about as many primes $p$ whose size is on the order of $a$ as you would expect. Our situation is more complicated because the values of $k_2\in\mathcal{K}$ are not prime. On the other hand, our would-be generating set is now much larger than just a single integer. Goldfeld's argument has two main components, a simple combinatorial argument to deal with small subgroups and an argument based on the large sieve inequality to deal with low index subgroups.

Our argument has the same basic outline. We use a version of the larger sieve (which is combinatorial in nature) to deal with all subgroups except those with bounded index. As a result, when we get to the large sieve inequality, we only need to deal with characters of bounded order. However, we need a sharper inequality than the typical large sieve inequality offers. Therefore, the first thing we do is prove a version of the large sieve inequality for characters of fixed order which is optimized for our purposes.

Before continuing, we give an explanation of why the large sieve by itself does not seem to be sufficient, despite appearing on first glance to be perfectly suited for our problem. Suppose we fix $k_1$ and try to show that $\mathcal{P}^{\1}_{k_1}$ is not concentrated in the kernels of many different characters (corresponding to values of $k_2$). If we try to immediately apply the large sieve inequality (in the form of Lemma \ref{largesieve}), the $N^2+M$ term presents a significant obstacle. 
There are two underlying issues. One is that  $N$ ``bad'' characters are enough to ruin all values of $k_2$.  The other is that $|\text{Supp}\,a_m|$will be a small power of $M$. The first is problematic when $N$ is large, and otherwise, 
the second becomes problematic; essentially, there is not enough averaging to exploit.

To get around the first issue, we will restrict our focus to characters of small order when applying the large sieve inequality. While there are around $N^2$ characters modulo integers less than $N$, there are only around $N$ such characters of small order. Therefore, if we only sum over this smaller number of characters, we can hope to replace the $N^2$ term in the large sieve inequality with $N$. This having been done, we will then be able to take $N$ large enough to obtain a version of the large sieve inequality that we can profitably apply.

Recall that we set $R_1=\log x/(\log\log x)^3$. We now take $r<R_1$ to be a prime, and we set $\xi_r=e^{2\pi i/r}$. We can split the prime ideals of $\mathbb{Z}[\xi_r]$ into two categories, those of degree 1 (corresponding to rational primes that split completely) and those of degree $>1$.

We say that a function on ideals of $\mathbb{Z}[\xi_r]$ is \emph{semi-multiplicative} if it is multiplicative for ideals whose norms are relatively prime.
Define the semi-multiplicative function $\nu$ by setting its value on every ideal $I$ of norm $p^i$ for $i\ge 1$ as follows. If the prime factors of $I$ have degree $d>1$ then $\nu(I)=(-1)^{i/d}$ if $I$ is square-free
and is zero otherwise.  If the prime factors $\{\frak{p}_j\}_{j=1}^i$ of $I$ have degree $1$
then $\nu(I) = (-1)^{i-1}(i-1)$ if $I$ is a product of $i$ distinct prime ideals, $\nu(I) = (-1)^{i-1}$ if $I$ has $i-1$ distinct prime factors, and 
$\nu(I)=0$ if $I$ has $\le i-2$ distinct prime factors.  This definition is motivated by the following lemma.

\begin{lemma}For $\frak{b}$ relatively prime to $(r)$,
$$\sum_{\frak{d}\mid \frak{b}} \nu(\frak{d}) = \mu^2(N(\frak{b}))
=\begin{cases}
1&\text{if }N(\frak{b}) \text{ is square-free}\\
0&\text{else.}
\end{cases}$$
\end{lemma}
\begin{proof}
It suffices to check this when $N(\frak{b})$ is the power of some prime $p\neq r$.
Suppose $p$ does not split completely.  Let $\mathfrak{b} = \prod_{i=1}^k \mathfrak{p}_i^{a_i}$.  Then
$$\sum_{\mathfrak{d}\mid \mathfrak{b}} \nu(\mathfrak{d}) = \sum_{b_i\in \{0,1\}} \nu\Bigl(\prod_{i=1}^k \mathfrak{p}_i^{b_i}\Bigr) = \prod_{i=1}^k (1+\nu(\mathfrak{p}_i)) = 0 = \mu^2(N(\mathfrak{b})).$$
We therefore assume that $p$ does split completely.
\begin{align*}
\sum_{\mathfrak{d}\mid \mathfrak{b}} \nu(\mathfrak{d}) &= \sum_{b_i\in \{0,1\}} \nu\Bigl(\prod_{i=1}^k \mathfrak{p}_i^{b_i}\Bigr)+
\sum_{i:a_i\ge 2} \sum_{b_l\in \{0,1\}\,\forall l\neq i} \nu(P_1^{b_1}\cdots P_i^2\cdots P_k^{b_k}) \\
&= 1 + \sum_{j=2}^k (-1)^{j-1}(j-1)\binom kj - \sum_{i: a_i\ge 2} \prod_{l\neq i} (1-1).
 \end{align*}
If $k=1$ and $a_1\ge 2$, this is $0$.  If $k=a_1=1$ then it is $1$.
Finally, if $k\ge 2$, it is
\begin{align*}
1 + \sum_{j=2}^k (-1)^{j-1}(j-1)\binom kj & = 1 + \sum_{j=2}^k (-1)^{j-1}(j-1)\binom kj  - \sum_{j=0}^k (-1)^j \binom kj\\
& = -\sum_{j=0}^k j (-1)^j \binom kj = -\frac d{dx}(1-x)^k|_{x=1}= 0.
\end{align*}

\end{proof}

\def\Z{\mathbb{Z}}
\def\fd{\mathfrak{d}}

\begin{proposition}
\label{zeta}
Let $r$ be a prime.  Let
$$\zeta_\nu(s) = \sum_{\substack{\fd\subset\Z[\xi_r]\\ r\nmid N(\fd)}} |\nu(\fd)| N(\fd)^{-s}$$
\begin{enumerate}
\item[(a)] There exists $C$ such that if $r\ge 5$ then
$$\zeta_\nu(1-\frac 2r) \le C^r.$$
\item[(b)] If $r=3$ then $\zeta_\nu(s)$ converges for $s>1/2$.
\end{enumerate}
\end{proposition}

\begin{proof}
For $\Re(s)$ sufficiently large, we have an Euler product decomposition
$$\zeta_\nu(s) = \prod_p L_p(p^{-s}),$$
so
\begin{equation}
\label{log sum}
\log \zeta_\nu(s) = \sum_p \log L_p(p^{-s}).
\end{equation}
For part (a), it suffices to prove that if $r\ge 5$, the right hand side of \eqref{log sum} converges for $s=1-2/r$ and has value $O(r)$; for part (b), it suffices to prove the sum converges for $s>1/2$.

We first assume $r\ge 5$.
If $p$ is of order $d>1$ (mod $r$) then
$$L_p(x) = (1 + x^d)^{\frac{r-1}d},$$
and therefore
$$\log L_p(x) \le \frac{r-1}d x^d \le \frac{r-1}2 x^2.$$
As $4/r-2<-1$, $\sum_p p^{4/r-2} = O(1)$, and we have
$$\sum_{p\not\equiv 1\pmod r} \log L_p(p^{2/r-1}) = O(r).$$

For $p\equiv 1\pmod r$, we have
$$L_p(x) = 1+\sum_{k=2}^\infty \biggl((k-1)\binom {r-1}k + (r-1)\binom{r-2}{k-2}\biggr) x^k.$$
Therefore,
\begin{align*}
\log L_p(p^{2/r-1}) < \sum_{k=2}^r \frac{r^k p^{-k+2k/r}}{k(k-2)!} + \sum_{k=2}^r \frac{r^{k-1} p^{-k+2r/k}}{(k-2)!} &< 2\sum_{k=2}^r \frac{r^k p^{-k+2k/r}}{k(k-2)!} \\
&< 4\sum_{k=2}^\infty \frac{r^k p^{-k+2k/r}}{(k-1)!}.
\end{align*}
Now, $p>r$, so 
$$r p^{-1+2/r} = r^{2/r} (p/r)^{-1+2/r}  < r^{2/r} < 2,$$
and
$$\log L_p(p^{2/r-1}) < 4\sum_{k=2}^\infty \frac{2^k}{(k-1)!}  = 8(e^2-1)$$
for all $p$.  Also, there are fewer than $r$ primes $p\equiv 1\pmod r$ which are less than $r^2$, so the contribution of all such primes to
\eqref{log sum} is $O(r)$.
If $p\ge r^2$, we have $r < p^{1-2/r}$, so
$$\log L_p(p^{2/r-1})\le 4 r^2 p^{-2+4/r}\sum_{k=2}^\infty \frac 1{(k-1)!}.$$
It therefore suffices to prove that 
$$\sum_{\substack{p>r^2\\ p\equiv 1\pmod r}}r^2 p^{-2+4/r} = O(r).$$
Since the $n$th prime $p>r^2$ which is $1$ (mod $r$) is greater than $r(r+n-1)$, the left-hand side is less than
$$r^{4/r} \sum_{n=0}^\infty (r+n-1)^{-2+4/r} = O(r^{-1/5}),$$
so in fact much more is true than we need.

Finally, we consider the $r=3$ case.
Using our previous estimates,
$$\log L_p(x) \le 3x^2+2x^3.$$
Summing this expression as $x$ ranges over $p^{-1/2-\epsilon}$, where $p$ is prime, or even $n^{-1/2-\epsilon}$ where $n$ is a positive integer, it does, indeed, converge.
\end{proof}

We remark that part (b) of this proposition implies that 
\begin{align*}
\sum_{\substack{\fd\subset\Z[\zeta_3]\\ 3\nmid N(\fd) \le Q}} |\nu(\fd)| N(\fd)^{-1/3} & \le Q^{1/4} \sum_{\substack{\fd\subset\Z[\zeta_3]\\ 3\nmid N(\fd)
 \le Q}} |\nu(\fd)|  N(\fd)^{-7/12} \\
 &\le Q^{1/4}\zeta_\nu(7/12) \ll Q^{1/4}.
\end{align*}

I would like to thank Michael Larsen for suggesting Proposition~\ref{zeta}.

We now state a proposition which takes inspiration from a result of Elliott \cite{E:1970} about quadratic characters.
\begin{proposition}\label{elliott}
Let $\mathcal{M}$ be a set of $r$th-power-free positive integers, all with the same $2$-adic valuation. Let $M$ be the largest element of $\mathcal{M}$. 
Let $\mathcal{Q}$ be a finite set of positive integers coprime to $r$, and let $Q$ be the largest element of $\mathcal{Q}$.  We assume
$Q$ is greater than $x^r$ and $M$ is greater than $x$. Then $$\sum_{q\in \mathcal{Q}}\sum_{\substack{\chi\emph{ mod }q\\ \chi^r=\chi_0}}^{\star}\Big\vert\sum_{m\in\mathcal{M}}\chi(m)\Big\vert^2\ll Q^{1-\frac{1}{r}}M^4+Q'\vert \mathcal{M}\vert,$$
where $Q'=\#\{\chi\emph{ mod }q: \chi^r=\chi_0, q\in \mathcal{Q}\}$.
\end{proposition}
\begin{proof}
To begin with, \begin{align}\sum_{q\in \mathcal{Q}}\sum_{\substack{\chi\text{ mod }q\\ \chi^r=\chi_0}}^{\star}\Big\vert\sum_{m\in\mathcal{M}}\chi(m)\Big\vert^2\notag&=\sum_{m_1,m_2\in\mathcal{M}}\sum_{q\in \mathcal{Q}}\sum_{\substack{\chi\text{ mod }q\\ \chi^r=\chi_0}}^{\star}\chi(m_1)\overline{\chi(m_2)}\\&\ll \vert\mathcal{M}\vert Q'+\sum_{\substack{m_1,m_2\in\mathcal{M}\\ m_1\neq m_2}}\sum_{q\in \mathcal{Q}}\sum_{\substack{\chi\text{ mod }q\\ \chi^r=\chi_0}}^{\star}\chi(m_1)\overline{\chi(m_2)}.\label{secondterm}\end{align} 
We recall the
$r$th power symbol (whose definition and properties are given on pages 348 and 349 of \cite{CF:1967}). 
Proposition 3.2 of \cite{BR:2023} tells us that the set of primitive characters of order $r$ which have conductor relatively prime to $r$ is in bijection with ideals of $\mathbb{Z}[\xi_r]$  whose norms are square-free
and coprime to $r$. In this correspondence, the conductor of the character equals the norm of the corresponding ideal.  Moreover, the characters can be thought of as $r$th power symbols.
With this in mind, we can bound the second term of \eqref{secondterm} by 
\begin{equation}\label{repeat}\sum_{\substack{m_1,m_2\in\mathcal{M}\\ m_1\neq m_2}}
\sum_{\substack{\mathfrak{b}\subset\mathbb{Z}[\xi_r]\\ r\nmid N(\mathfrak{b})\leq Q\\ \mathfrak{b}+(m_1m_2)=(1)}}\mu^2(N(\mathfrak{b}))\Big(\frac{m_1}{\mathfrak{b}}\Big)_r\Big(\frac{m_2}{\mathfrak{b}}\Big)_r^{-1}.\end{equation} 

Note that $\big(\frac{m_1}{\boldsymbol{\cdot}}\big)_r$ is a multiplicative function. Moreover, we claim that if $b\equiv 1\lessspacepmod{2r^2m_1}$ then $(b)$ is in the kernel of this character. Indeed, this is a consequence of Exercise 1.8 from \cite{CF:1967}, along with the fact that anything congruent to 1 mod $2^{j+2}$ is a $2^j$-th power in $\mathbb{Q}_2$ and anything congruent to 1 mod $p^{j+1}$ is a $p^j$-th power in $\mathbb{Q}_p$ for any odd prime $p$. Thus $\big(\frac{m_1}{\boldsymbol{\cdot}}\big)_r$ is an ideal character whose modulus divides $(2r^2m_1)$, and similarly $\big(\frac{m_2}{\boldsymbol{\cdot}}\big)_r$ is an ideal character whose modulus divides $(2r^2m_2)$.

Now, $$\Big(\frac{m_1}{\boldsymbol{\cdot}}\Big)_r=\Big(\frac{m_2}{\boldsymbol{\cdot}}\Big)_r$$ for values of $\boldsymbol{\cdot}$ prime to $m_1$ and $m_2$ if and only if $\frac{m_1}{m_2}$ is an $r$th power in $(\mathbb{Q}(\xi_r))_{\mathfrak{p}}$ for each $\mathfrak{p}$ not containing $r$, $m_1$, or $m_2$. (See, for instance, Exercise 1.5 of \cite{CF:1967}.) Since $\frac{m_1}{m_2}$ has 2-adic valuation 0, the Grunwald-Wang theorem \cite{W:1950} tells us that this is the case if and only if $\frac{m_1}{m_2}$ is a rational $r$th power. In particular, for $m_1\neq m_2$, we conclude that $\big(\frac{m_1}{\boldsymbol{\cdot}}\big)_r\big(\frac{m_2}{\boldsymbol{\cdot}}\big)_r^{-1}$ is a non-trivial ideal character modulo $\mathfrak{c}_{m_1,m_2}$, where $\mathfrak{c}_{m_1,m_2}$ is some ideal dividing $(2r^2m_1m_2)$. Let $\chi_{m_1,m_2}$ be the primitive character corresponding to this character.

A theorem of Landau \cite{L:1918} says that if $\chi$ is a primitive ideal character modulo an ideal $\mathfrak{q}$ in an number ring of degree $n$ then $$\sum_{N(\mathfrak{a})\leq z}\chi(\mathfrak{a})\ll N(\mathfrak{q})^{\frac{1}{n+1}}\log^n(N(\mathfrak{q}))z^{\frac{n-1}{n+1}}.$$ This appears as Theorem 2 in \cite{S:1972}, which shows that this inequality can be made explicit. We would like to apply Landau's theorem to \eqref{repeat}, but we first need to deal with $\mu^2$. To do this, we use an inclusion-exclusion approach similar to the one taken by Munsch in \cite{M:2014}.

Let $\mu_{m_1,m_2}$ be the multiplicative function whose support is minimal, subject to the condition that $\mu_{m_1,m_2}(\mathfrak{p})=-1$ if $\frak{c}_{m_1,m_2}\subset \mathfrak{p}$. Note that 
$$\mu^2(N(\mathfrak{b}))\mathbbm{1}_{\mathfrak{b} + \mathfrak{c}_{m_1,m_2}= (1)}
=\sum_{\frak{d}\mid\mathfrak{b}}\nu(\frak d)\sum_{\mathfrak{a}\mid \mathfrak{b}}\mu_{m_1,m_2}(\mathfrak{a}).$$ 
Using this identity and swapping the order of summation, we can rewrite \eqref{repeat} as 
$$\sum_{\substack{m_1,m_2\in\mathcal{M}\\ m_1\neq m_2}}
\sum_{\frak{d}\subset \mathbb{Z}[\xi_r]}\nu(\frak d)
\sum_{\mathfrak{a}\subset\mathbb{Z}[\xi_r]}\mu_{m_1,m_2}(\mathfrak{a})
\sum_{\substack{\mathfrak{b}\subset \mathfrak{d}\cap\mathfrak{a}\\ r\nmid N(\mathfrak{b})\leq Q}}\chi_{m_1,m_2}(\mathfrak{b}).$$ 
Certainly $|\text{Supp}\,\mu_{m_1,m_2}| = M^{o(1)}$, so we can bound the previous expression by
\begin{equation}\label{nextstep}
M^{o(1)}\sum_{\substack{m_1,m_2\in\mathcal{M}\\ m_1\neq m_2}}
\sum_{\mathfrak{d}}
|\nu(\mathfrak{d})|
\Big\vert\sum_{\substack{\mathfrak{b}\subset \mathfrak{d}\\ r\nmid N(\mathfrak{b})\leq Q}}\chi_{m_1,m_2}(\mathfrak{b})\Big\vert.\end{equation} 
By Landau's theorem, the innermost sum can be asymptotically bounded by $$M^2\log^{r-1}(M^{2(r-1)})\Big(\frac{Q}{N(\mathfrak{d})}\Big)^{\frac{r-2}{r}}.$$ As $2r<\frac{\log x}{\log\log x}<\frac{\log M}{\log\log M}$, we have $$\log_M\big(\log^{r-1}(M^{2(r-1)})\big)<\log_M\big((2r\log M)^r\big)<2r\frac{\log \log M}{\log M}<1$$ 
so this in turn can be asymptotically bounded by $$M^3Q^{1-\frac{2}{r}}N(\mathfrak{d})^{\frac{2}{r}-1}.$$ 
By Proposition~\ref{zeta},  \eqref{nextstep} can be bounded by
$$M^{3+o(1)}\zeta_\nu(1-2/r)Q^{1-2/r} = M^{3+o(1)}Q^{1-2/r}$$
or, in the case $r=3$, by
$$M^{3+o(1)}Q^{7/12}.$$
Putting everything together, 
$$\sum_{q\in \mathcal{Q}}\sum_{\substack{\chi\text{ mod }q\\ \chi^r=\chi_0}}^{\star}\Big\vert\sum_{m\in\mathcal{M}}\chi(m)\Big\vert^2\ll \vert\mathcal{M}\vert Q'+M^4Q^{1/r}.$$ 
\end{proof}
It is instructive to compare this result to Theorem 1.6 of \cite{BR:2023}. Balestrieri and Rome use a much deeper method of proof to arrive at a similar-looking statement but with several subtle differences. To begin with, they require the elements of $\mathcal{M}$ to be square-free. This is an advantage to our method, albeit one that is irrelevant in the context we will use it in. Under this condition, the theorem of Balestrieri and Rome in particular implies that \begin{equation}\label{Balestrieri-Rome}\sum_{q\leq Q}\sum_{\substack{\chi\text{ mod }q\\ \chi^r=\chi_0\\ \chi\neq \chi_0}}\Big\vert\sum_{m\in\mathcal{M}}\chi(m)\Big\vert^2\ll Q^{1+o(1)}+M^{O(r)}.\end{equation} Note that the characters are no longer restricted to being primitive. This is because Balestrieri and Rome are able to obtain a much stronger bound when $Q$ is of comparable size to $M$, which becomes relevant when one considers the conductors of non-primitive characters. On the other hand, our result gives a stronger bound in the case where $Q$ is considerably larger than $M$.

Fix $k_1\in\mathcal{K}_1^{\star}$. Let $\mathcal{K}_0$ be the set of integers less than $x$ which are not divisible by 4 and whose odd prime factors are greater than $x^{\iota}$. Let $\mathcal{X}$ be the set of primitive order $r$ characters whose conductors lie in $\mathcal{K}_0$ such that $\chi(p)$ is constant on a subset of $\mathcal{P}^{\1}_{k_1}$ containing at least $(1-\frac{\log\log x}{\log x})|\mathcal{P}^{\1}_{k_1}|$ elements.

Suppose that for a proportion of values of $k_2\in \mathcal{K}_2^{\star}$ which is bounded from below by a positive constant, there is an index $r$ coset of $(\mathbb{Z}/k_2\mathbb{Z})^{\times}$ for which the proportion of $\mathcal{P}^{\1}_{k_1}$ not contained in the coset is less than $\frac{1}{\log x}$. There is then some $\eta\geq \iota$ for which there are $\frac{x^{\eta}}{\log^{O(1)}x}$ elements in $\mathcal{X}$ whose conductors are less than $x^{\eta}$. Fix such an $\eta$, and let $\mathcal{Y}$ be the set of such elements.

Let $\alpha$ be the smallest multiple of $\eta$ greater than $r\log\log x$. Let $Q=x^{\alpha}$. 
Let $\mathcal{Z}$ be the set of order $r$ primitive characters with conductor $<Q$ such that $\chi(p)$ is constant on a subset of $\mathcal{P}^{\1}_{k_1}$ containing almost all of the elements. 
Clearly,
\begin{equation}
\label{obvious}
\sum_{\chi\in \mathcal{Z}}\Big\vert\sum_{p\in \mathcal{P}^{\1}_{k_1}}\chi(p)\Big\vert^2\gg \vert\mathcal{Z}\vert\vert\mathcal{P}^{\1}_{k_1}\vert^2.
\end{equation}

We now wish to bound  the size of $\mathcal{Z}$ from below.  We consider characters $\chi$ of the form $\prod_{i=1}^n\chi_i$, where $n=\alpha/\eta$ and each $\chi_i$ is an element of $\mathcal{Y}$. Considering the size of $n$, we see that $\chi(p)$ is constant on a subset of $\mathcal{P}^{\1}_{k_1}$ containing almost all of the elements. Moreover, if the conductors $q_1,\ldots,q_n$ of the characters have the property that $(q_i,q_j)>2$ only if $i=j$, then $\chi$ will be non-trivial, so in this case
$\chi$ will be an element of $\mathcal{Z}$.

Note that for any $q$ which is the conductor of some element in $\mathcal{Y}$, there are at most $x^{\eta-\iota+o(1)}$ elements of $\mathcal{Y}$ whose conductors $q'$ satisfy $(q,q')>2$. Let $Y=\vert\mathcal{Y}\vert$. Then there are at least 
$$\binom{Y}{n}-x^{2\eta-\iota+o(1)}\binom{Y}{n-2} > \frac 12\binom Yn$$ 
products which result in elements of $\mathcal{Z}$. There is another consideration, however.
Even ignoring the order of the terms in the product, it is possible for our product representations to fail to be unique.
Indeed, this is because the elements of $\mathcal{Y}$ have conductors which may not be prime.  Fortunately, the number of representations of
each element of $\mathcal{Z}$ is bounded by
$$(n/\iota)!=(\log x)^{O(\log x/(\log\log x)^2)}=x^{o(1)}.$$
Note that $$\binom{Y}{n}=\big(\frac{x^{\eta}}{\log^{O(1)}x}\big)^n=Q(\log x)^{-O(\log x/(\log\log x)^2)}=Qx^{-o(1)}.$$ Putting everything together, we see that $\mathcal{Z}$ has $x^{-o(1)}Q$ elements.

However, from Proposition \ref{elliott}, we have 
$$\sum_{\chi\in \mathcal{Z}}\Big\vert\sum_{p\in \mathcal{P}^{\1}_{k_1}}\chi(p)\Big\vert^2\ll \vert\mathcal{Z}\vert\vert\mathcal{P}^{\1}_{k_1}\vert.$$
Comparing this with \eqref{obvious}, we obtain the contradiction
$\vert \mathcal{P}^{\1}_{k_1}\vert=O(1)$. Our supposition regarding cosets of index $r$ is therefore disproven. 

We turn now to smaller subgroups, where crude equidistribution estimates will suffice.
\begin{lemma}\label{large}
Let $\mathcal{B}$ be a set of positive integers less than some integer $B$. Then 
$$\sum_{k\in\mathcal{K}}\sum_{a\emph{ mod }k}\#\{b\in\mathcal{B}: b\equiv a\lessspacepmod{k}\}^2\ll \vert\mathcal{K}\vert\vert \mathcal{B}\vert+\Big(\frac{\log B}{\log x}\Big)^{1/\iota}\vert \mathcal{B}\vert^2.$$
\end{lemma}
\begin{proof}
Suppose $(k,b,b')$ is a triple of integers with $k\in\mathcal{K}$ and $b$ and $b'$ distinct elements of $\mathcal{B}$ congruent to each other mod $k$. We must then have $k\mid b-b'$. Now, $b-b'$ can have $O(\log_xB)$ prime factors greater than $x^{\iota}$. For $b$ and $b'$ distinct, there are at most $O(\log^{1/\iota}_xB)$ values of $k$ such that $(k,b,b')$ is a triple with the aforementioned properties. 
\end{proof}
This lemma becomes especially interesting when $|\mathcal{B}|\ge |\mathcal{K}|$.  Consequently, 
we will take $\mathcal{B}$ to be the set of integers which can be written as a product of $10$ elements of $\mathcal{P}^{\1}_{k_1}$. We say that $\mathcal{P}^{\1}_{k_1}$ covers $(\mathbb{Z}/k_2\mathbb{Z})^{\times}$ sufficiently if for any set of $\lfloor\frac{k_2}{R_2}\rfloor$ residues mod $k_2$, almost none of the elements of $\mathcal{B}$ have residues mod $k_2$ in that set.

Suppose it is not true that for almost all $k_2\in\mathcal{K}^{\star}_2$, the set $\mathcal{P}^{\1}_{k_1}$ covers $(\mathbb{Z}/k_2\mathbb{Z})^{\times}$ sufficiently. Then we see that $$\sum_{k\in\mathcal{K}}\sum_{a\text{ mod }k}\#\{b\in\mathcal{B}: b\equiv a\lessspacepmod{k}\}^2\gg \vert\mathcal{K}_2^{\star}\vert\frac{x}{R_2}\frac{\vert\mathcal{B}\vert^2}{(x/R_2)^2}\gg \frac{R_2\vert\mathcal{B}\vert^2}{\log x}.$$ On the other hand, $\vert\mathcal{B}\vert>x$ and $\log_xb=O(1)$ for all $b\in\mathcal{B}$ so Lemma \ref{large} tells us that $$\sum_{k\in\mathcal{K}}\sum_{a\text{ mod }k}\#\{b\in\mathcal{B}: b\equiv a\lessspacepmod{k}\}^2\ll \vert \mathcal{B}\vert^2.$$ This gives a contradiction since it is not true that $R_2\ll \log x$.

If $\mathcal{M}$ is a set of square-free positive integers bounded by $M$ and $Q$ is a positive integer then \begin{equation}\label{applied}\sum_{q\leq Q}\sum_{\substack{\chi\text{ mod }q\\ \chi\neq \chi_0\\ \chi^{2^T}=\chi_0}}\Big\vert\sum_{m\in\mathcal{M}}\chi(m)\Big\vert^2\ll Q^{1+o(1)}+M^{O(1)}\end{equation} by applying \eqref{Balestrieri-Rome} with $r=2^T$.

We claim that for almost every $k_2$, there exists $A\in \mathcal{A}^{\1}_{k_1}$ such that $m_0A$ is odd order mod $k_2$. Suppose there are at least $\frac{x}{\log^2 x}$ values of $k_2$ for which this is not the case. Let $\mathcal{Q}(j)$ be the set of $j$-fold products of these values of $k_2$. Take $j_0$ to be the smallest value for which $$\vert \mathcal{Q}(j_0)\vert\geq \vert\mathcal{A}^{\1}_{k_1}\vert.$$ Then take $j'_0$ to be $j_0$ multiplied by some very large constant like $T^T$. Note that $$\vert\mathcal{A}^{\1}_{k_1}\vert\geq \binom{\vert\mathcal{P}^{\1}_{k_1}\vert}{J}\big/x^{o(J)}=\vert\mathcal{P}^{\1}_{k_1}\vert^{J(1-o(1))},$$ where the factor of $x^{o(J)}$ appears as the number of $J$ element subsets of $\mathcal{P}^{\1}_{k_1}$ containing the prime factors of any particular element of $\mathcal{A}^{\1}_{k_1}$. Remembering that $$\log\vert\mathcal{P}^{\1}_{k_1}\vert\gg \log x,$$ we then have the important asymptotic $$\log\vert\mathcal{A}^{\1}_{k_1}\vert\gg \log\max_{A\in\mathcal{A}^{\1}_{k_1}} A.$$ The conclusion of all this is that $$\frac{\log\vert Q(j'_0)\vert}{\log\max_{A\in\mathcal{A}^{\1}_{k_1}} A}$$ is larger than the smallest value of $O(1)$ one can take in \eqref{applied} while still being bounded. Note that $j'_0$ is the same order of magnitude as $J$. Taking $Q=Q(j'_0)$, one easily has the bound $\vert \mathcal{Q}\vert=Q^{1-o(1)}$, where $Q=\max_{q\in\mathcal{Q}}q$.

Note that $m_0A$ having even order mod $n$ is equivalent to the sum $$\sum_{\substack{\chi\text{ mod }n\\ \text{power of 2 order}}}\chi(m_0A)$$ being equal to 0. It thus follows that $$\sum_{q\in\mathcal{Q}}\sum_{A\in \mathcal{A}^{\1}_{k_1}}\sum_{\substack{\chi\text{ mod }q\\ \text{power of 2 order}}}\chi(m_0A)=0.$$ Moving the principal characters to the right-hand side, we obtain $$\sum_{q\in\mathcal{Q}}\sum_{\substack{\chi\text{ mod }q\\ \text{power of 2 order}\\ \chi\neq\chi_0}}\sum_{A\in \mathcal{A}^{\1}_{k_1}}\chi(m_0A)=-\vert \mathcal{Q}\vert\vert \mathcal{A}^{\1}_{k_1}\vert.$$ By the Cauchy-Schwarz inequality, $$\sum_{q\in\mathcal{Q}}\sum_{\substack{\chi\text{ mod }q\\ \text{power of 2 order}\\ \chi\neq\chi_0}}\Big\vert\sum_{A\in \mathcal{A}^{\1}_{k_1}}\chi(m_0A)\Big\vert^2\gg 2^{-O(j'_0)}\vert \mathcal{Q}\vert\vert \mathcal{A}^{\1}_{k_1}\vert^2,$$ yet this contradicts \eqref{applied}.

Choose random values of $k_1\in\mathcal{K}^{\star}_1$ and $k_2\in\mathcal{K}^{\star}_2$. Then $k_1$ and $k_2$ are relatively prime to $L$. Moreover, the following are true with probability 1:

\begin{itemize}
\item For every non-principal character $\chi$ mod $k_2$ and for any complex number $z$,
\begin{equation}\label{kequi}
\#\{p\in\mathcal{P}^{\1}_{k_1}: \chi(p)\neq z\}\geq \frac{\vert\mathcal{P}^{\1}_{k_1}\vert}{\log x}
\end{equation}
\item $\mathcal{P}_{k_1}^{\1}$ covers $(\mathbb{Z}/k_2\mathbb{Z})^{\times}$ sufficiently
\item For every non-principal character $\chi$ mod $k_1$ and for any complex number $z$,
\begin{equation*}
\#\{p\in\mathcal{P}^{\2}_{k_2}: \chi(p)\neq z\}\geq \frac{\vert\mathcal{P}^{\2}_{k_2}\vert}{\log x}
\end{equation*}
\item $\mathcal{P}_{k_2}^{\2}$ covers $(\mathbb{Z}/k_1\mathbb{Z})^{\times}$ sufficiently
\item $m_0A^{\1}_{k_1}$ has odd order mod $k_2$ for some $A^{\1}_{k_1}\in\mathcal{A}^{\1}_{k_1}$
\item $m_0A^{\2}_{k_2}$ has odd order mod $k_1$ for some $A^{\2}_{k_2}\in\mathcal{A}^{\2}_{k_2}$
\end{itemize}

In the next section, we will assume that we have values of $k_1$, $k_2$, $A^{\1}_{k_1}$, and $A^{\2}_{k_2}$ which satisfy these properties.
\section{Constructing Products with Specified Residues}\label{extra}
Let $\mathcal{P}^{\star}$ be a randomly chosen subset of $\mathcal{P}^{\1}_{k_1}$ with $J$ elements. Let $G$ be the group $(\mathbb{Z}/L\mathbb{Z})^{\times}\oplus (\mathbb{Z}/k_2\mathbb{Z})^{\times}\oplus \mathbb{Z}/\ell_0\mathbb{Z}$. We can think of the elements of $\mathcal{P}^{\star}$ as elements in $G$ under the embedding $(p\lessspacepmod{L},p\lessspacepmod{k_2},1)$. We can decompose $G$ as $G_1\oplus G_2\oplus G_3\oplus H$, where $G_2$ is determined uniquely up to isomorphism subject to the conditions that  $G_1$ is the maximal odd order subgroup of $(\mathbb{Z}/L\mathbb{Z})^{\times}$, $G_3$ is the maximal subgroup of $\mathbb{Z}/\ell_0\mathbb{Z}$ whose order is relatively prime to $2\phi(k_2)$, and $H$ is the $2$-Sylow subgroup of $G$. In particular, these four components have relatively prime orders.

For $i\in\{1,2,3,4\}$, take $\chi_i$ to be a character of $G_i$ (where $G_4$ means $H$). Note that the order of $\chi_4$ must be $O(1)$ since no prime factor of $Lk_2$ is congruent to 1 mod $2^T$. Let $r$ be the order of $\chi_2$. If $r=o(\lambda(L))$ then $\vert\arg\chi_1(p)\chi_2(p)\chi_3(p)\chi_4(p)\vert < \frac{1}{\lambda(L)^2}$ only if $\chi_2(p)=1$. Moreover, $\chi_2(p)\neq 1$ for at least $\frac{J}{\log^2 x}$ elements in $\mathcal{P}^{\star}$ with probability 1, assuming $r\neq 1$. Indeed, this is based on \eqref{kequi} and the fact that the component of $\chi_2$ associated with $\mathbb{Z}/\ell_0\mathbb{Z}$ has the same value on every $p$. On the other hand, if $r\gg \lambda(L)$ then the inequality 
$$\vert\arg\chi_1(p)\chi_2(p)\chi_3(p)\chi_4(p)\vert < \frac{1}{\lambda(L)^2}$$ 
implies that the residue of $p$ mod $k_2$ is one of $O(\frac{k_2}{\lambda(L)})$ possibilities. Moreover, due to the multiplicative structure of these possibilities, if almost all the elements in $\mathcal{P}^{\1}_{k_1}$ correspond to such a residue then for any positive integer $n$ of fixed size, there is a list of $O(\frac{k_2}{\lambda(L)})$ residues mod $k_2$ such that almost all products of $n$ elements of $\mathcal{P}^{\1}_{k_1}$ have residue mod $k_2$ in the list. Recall that $\mathcal{P}_{k_1}^{\1}$ covers $(\mathbb{Z}/k_2\mathbb{Z})^{\times}$ sufficiently. Thus, for the residue of $p$ mod $k_2$ to be one of $O(\frac{k_2}{\lambda(L)})$ possibilities for all but at most $\frac{J}{\log^2 x}$ elements of $\mathcal{P}^{\star}$ has probability 0. We conclude, under the assumption that $\chi_2$ is non-trivial, that for almost all choices of $\mathcal{P}^{\star}$, $\vert\arg\chi_1(p)\chi_2(p)\chi_3(p)\chi_4(p)\vert\geq \frac{1}{\lambda(L)^2}$ for at least $\frac{J}{\log^2 x}$ elements in $\mathcal{P}^{\star}$ (with $\chi_1$, $\chi_2$, $\chi_3$, and $\chi_4$ as before).

\begin{lemma}\label{hittarget}
Let $g\in G$ be 1 mod $L$, odd order mod $k_2$, and $w_0$ mod $\ell_0$. For almost all $J$ element subsets $\mathcal{P}^{\star}\subset \mathcal{P}^{\1}_{k_1}$, there is a product of distinct elements in $\mathcal{P}^{\star}$ equal to $g$ as an element of $G$.
\end{lemma}
\begin{proof}
Consider 
\begin{equation}\label{expr}\sum_{\chi_1\in\hat{G}_1}\sum_{\chi_2\in\hat{G}_2}\sum_{\chi_3\in\hat{G}_3}\sum_{\chi_4\in\hat{H}}
\overline{\chi_2(g)\chi_3(g)}\prod_{p\in\mathcal{P}^{\star}}\frac{1+\chi_1(p)\chi_2(p)\chi_3(p)\chi_4(p)}{2}.\end{equation} 
Note that we have omitted $\chi_1(g)$ and $\chi_4(g)$ since these two terms are equal to 1. Suppose $\chi_2$ is non-principal. 
By our previous discussion, it is true with probability 1 that 
\begin{align*}\prod_{p\in\mathcal{P}^{\star}}\frac{1+\chi_1(p)\chi_2(p)\chi_3(p)\chi_4(p)}{2}&\leq \Big\vert \frac{1+e^{i/\lambda(L)^2}}{2}\Big\vert^{\frac{J}{\log^2x}}=\Big\vert \cos\big(\frac{1}{2\lambda(L)^2}\big)\Big\vert^{\frac{J}{\log^2x}}\\&<\Big(1-\frac{1}{10\lambda(L)^4}\Big)^{\frac{J}{\log^2x}}\ll e^{-\lambda(L)^5}.\end{align*} 
We will assume this to be the case for the rest of the proof. We conclude that \eqref{expr} is equal to 
$$\sum_{\chi_1\in\hat{G}_1}\sum_{\chi_3\in\hat{G}_3}\sum_{\chi_4\in\hat{H}}\overline{\chi_3(g)}\prod_{p\in\mathcal{P}^{\star}}\frac{1+\chi_1(p)\chi_3(p)\chi_4(p)}{2}+O(\phi(k_2)\phi(L)\ell_0e^{-\lambda(L)^5}).$$ 
Moreover, if $\chi_3$ is non-trivial then $\vert\arg\chi_1(p)\chi_3(p)\chi_4(p)\vert\geq \frac{1}{\lambda(L)^2}$ for all $p$
since $\chi_3(p)$ is the value of a non-trivial character at a generator.
Hence \eqref{expr} is equal to 
$$\sum_{\chi_1\in\hat{G}_1}\sum_{\chi_4\in\hat{H}}\prod_{p\in\mathcal{P}^{\star}}\frac{1+\chi_1(p)\chi_4(p)}{2}+O(\phi(k_2)\phi(L)\ell_0e^{-\lambda(L)^5}).$$ 

The main term in this expression is equal to the number of products of distinct elements of $\mathcal{P}^{\star}$ corresponding to the identity in $G_1\oplus H$, multiplied by $2^{-J}\vert G_1\vert\vert H\vert$. By Lemma \ref{standard}, this can be bounded from below by $\frac{\vert G_1\vert\vert H\vert}{J^{\lambda(L)^2}}$. Therefore, \eqref{expr} is positive, from which we can conclude the existence of a product of distinct elements of $\mathcal{P}^{\star}$ which corresponds to $g$ as an element of $G$.
\end{proof}
\begin{proposition}\label{alteredAGPprop}
Let $G$ be a finite abelian group, let $g\in G$, and let $n$ be such that for any sequence of $n$ elements of $G$, there is a subsequence whose product is 1. Suppose $r>t>n$. Suppose $S$ is a sequence of length $r$ whose terms multiply to $g$. Then there are at least $\binom{r}{t}\binom{r}{n}^{-1}$ subsequences of $S$ of length at most $t$ whose product is $g$.
\end{proposition}
\begin{proof}
This is proven in the same way as Proposition 1.2 from \cite{AGP:1994}.
\end{proof}
We take $G$ as before, and we take $g$ to be 1 mod $L$, $(m_0A_{k_1}^{\1})^{-1}$ mod $k_2$, and $w_0$ mod $\ell_0$. 
We take $S$ to be a maximal sequence of distinct elements of $\{p\in \mathcal{P}^{\1}_{k_1}: p\nmid A^{\1}_{k_1}\}$ whose product corresponds to $g$ in $G$, constructed in the following manner.
We start off by finding a non-maximal sequence of this type by applying Lemma~\ref{hittarget}.
Then we take advantage of the fact that the Davenport constant of $G$ is smaller than $J$ to repeatedly add primes in groups whose products are the identity.
Consequently, 
$$|S|\ge \vert\mathcal{P}^{\1}_{k_1}\vert-\omega(A^{\1}_{k_1})-J.$$
Now, take $r$ to be the length of $S$, $t$ to be equal to $\lfloor r^{\epsilon}\rfloor$ for some small $\epsilon>0$, and $n$ to be equal to $J$. Let $\Pi_1^{\1},\ldots,\Pi_{N_1}^{\1}$ be the products of distinct elements of $S$ corresponding to $g$. By Proposition \ref{alteredAGPprop}, 
\begin{align*}\log N_1&\geq \log\bigg(\binom{\vert S\vert}{\lfloor\vert S\vert^{\epsilon}\rfloor}\binom{\vert S\vert}{J}^{-1}\bigg)>\lfloor\vert S\vert^{\epsilon}\rfloor(1-\epsilon)\log \vert  S\vert-J\log \vert S\vert\\&>\lfloor\vert S\vert^{\epsilon}\rfloor(1-2\epsilon)\log \vert S\vert
\end{align*} 
so 
$$\frac{\log N_1}{\log((y^wx)^{\lfloor\vert S\vert^{\epsilon}\rfloor})}>\frac{(1-2\epsilon)\log\vert S\vert}{(1+\frac{1}{\delta-3\iota})\log(y^w)}.$$ 
We have 
\begin{align*}\vert S\vert&\geq\vert\mathcal{P}^{\1}_{k_1}\vert-\omega(A^{\1}_{k_1})-J>\frac{\vert \mathcal{D}_1\vert}{2\log x\log\log x}\\&\geq \binom{\lceil y^{\theta+\rho}\log^{-\kappa-1}y\rceil}{w}(2\log x\log\log x)^{-1}\\&>\Big(\frac{y^{\theta+\rho}\log^{-\kappa-1}y}{w}\Big)^w(2\log x\log\log x)^{-1}.\end{align*} 
Therefore,
\begin{align*}\frac{\log N_1}{\log((y^wx)^{\lfloor\vert S\vert^{\epsilon}\rfloor})}&>(1-2\epsilon)\frac{w\log(w^{-1}y^{\theta+\rho}\log^{-\kappa-1}y)(1-o(1))}{w(1+\frac{1}{\delta-3\iota})\log y}\\&=(1-2\epsilon-o(1))\big(1+\frac{1}{\delta-3\iota}\big)^{-1}\rho(1-\iota).\end{align*} 
Taking $\epsilon$ and $\iota$ sufficiently small, this quantity approaches $\frac{1}{\denom}$.

Repeat everything with indices reversed, obtaining $\Pi_1^{\2},\ldots,\Pi_{N_2}^{\2}$ with analogous properties. We claim that $\Pi_0:=A^{\1}_{k_1}A^{\2}_{k_2}\Pi_{i_1}^{\1}\Pi_{i_2}^{\2}$ is a Carmichael seed for every $i_1$ and $i_2$. To begin with, each prime factor of $\Pi_0$ is congruent to $a_0$ mod $q_0$. Also, $\omega(\Pi_0)\equiv 4w_0\equiv 2\lessspacepmod{\ell_0}$.
We note that $A^{\1}_{k_1}$ and $\Pi_i^{\1}$ are 1 mod $k_1$, that $$\Pi_i^{\2}\equiv (m_0A^{\2}_{k_2})^{-1}\lessspacepmod{k_1},$$ that $$A^{\2}_{k_2}\equiv \frac{1}{m_0}\lessspacepmod{L_1},$$ and that $$A^{\1}_{k_1}\equiv \Pi^{\1}_{k_1}\equiv \Pi^{\2}_{k_2}\equiv 1\lessspacepmod{L_1}.$$ Thus $\Pi_0\equiv \frac{1}{m_0}\lessspacepmod{k_1}$ and $\Pi_0\equiv \frac{1}{m_0}\lessspacepmod{L_1}$. Since $(k_1,L_1)=1$, it follows that $\Pi_0\equiv \frac{1}{m_0}\lessspacepmod{k_1L_1}$. Similarly, $\Pi_0\equiv \frac{1}{m_0}\lessspacepmod{k_2L_2}$. Since for each prime $p\mid \Pi_0$, either $p-1\mid k_1L_1$ or $p-1\mid k_2L_2$, this concludes the proof that $\Pi_0$ is a Carmichael seed.

Let $$z=(y^wx)^{\vert \mathcal{P}^{\1}_{k_1}\vert^{\epsilon}+\vert \mathcal{P}^{\2}_{k_2}\vert^{\epsilon}} x^{3J}.$$ Then $\frac{\log(N_1N_2)}{\log z}$ approaches $\frac{1}{\denom}$. Therefore, we have produced $z^{\frac{1}{\denom}-o(1)}$ Carmichael seeds less than $z$, concluding the proof of Proposition \ref{workhorse}.
\section{Constructing Carmichael Numbers from Carmichael Seeds}\label{end}
Let $r$ mod $m$ be an arithmetic progression. Let $g=(r,m)$ and $h=(\lambda(g),m)$. Suppose $n$ is a Carmichael number congruent to $r$ mod $m$. If $p$ and $q$ are primes dividing $n$ with $p\mid q-1$ then $p\mid q-1\mid n-1$, which is a contradiction. Combining this with the well-known fact that Carmichael numbers must be odd and square-free, we conclude that $(n,2\phi(n))=1$. Consequently we need $(g,2\phi(g))=1$. We also need $\lambda(g)$ to divide $n-1$, so certainly the arithmetic progressions $1\lessspacepmod{\lambda(g)}$ and $r\lessspacepmod{m}$ must have non-empty intersection. Put another way, we need $r\equiv 1\lessspacepmod{h}$. Finally, suppose $36\mid m$ and $r\equiv 3\lessspacepmod{12}$. Then $n-1$ is divisible by neither 3 nor 4, so no prime dividing $\frac{n}{g}$ can be 1 mod 3 or 1 mod 4. Note that $g$ is divisible by neither 2 nor 9. Therefore, $12\mid \frac{m}{g}$, which implies that $\frac{r}{g}\equiv \frac{n}{g}\lessspacepmod{12}$. Since every prime dividing $\frac{n}{g}$ must be congruent to 11 mod 12, we conclude that $\frac{r}{g}\equiv 1\text{ or }11\lessspacepmod{12}$. To summarize, the following must all be false:

\begin{itemize}
\item $(g,2\phi(g))>1$
\item $h\nmid r-1$
\item $36\mid m$, $r\equiv 3\lessspacepmod{12}$, and $\frac{r}{g}\equiv 5\text{ or }7\lessspacepmod{12}$
\end{itemize}

This partially justifies our earlier terminology, in which we called an arithmetic progression $r\lessspacepmod{m}$ with any of these properties Carmichael incompatible. We now want to show that Carmichael compatible arithmetic progressions always do contain Carmichael numbers. We start with two small reduction steps.

\begin{lemma}
In proving Theorem~\ref{main}, we may assume that the progression $\frac rg \lessspacepmod{\frac mg}$ contains integers which are $\pm 1$ mod $12$.
 \end{lemma}

\begin{proof}
We may suppose $12\mid \frac mg$ as otherwise, no reduction is necessary.  Suppose $r-1$ is divisible by neither $3$ nor $4$.  Then $r\equiv 3\text{ or }11\lessspacepmod{12}$.
In the former case, $g$ is divisible by $3$, so $36\mid m$. By definition of Carmichael compatibility, $\frac rg\equiv 1\text{ or }11\lessspacepmod{12}$, again obviating the need for reduction.
The case $r\equiv 11\lessspacepmod{12}$ only requires reduction if $g\equiv 5\text{ or }7\lessspacepmod{12}$.
However, this would require $g$ to have a prime factor congruent to either $5$ or $7$ mod $12$, meaning that $\lambda(g)$ would be divisible by either $3$ or $4$.
This cannot be, since it would further imply that $h$, and consequently $r-1$, is divisible by either $3$ or $4$, violating our starting supposition.

Thus, we may restrict our attention to the case that $r-1$ is divisible by either 3 or 4. Let $s$ be equal to 6 in the former case and 4 in the latter case. Take $p$ to be a large prime congruent to either 5 or 7 mod 12 such that $\frac{p-1}{s}$ has no small prime factors. Note that $s=(p-1,r-1)$. To find Carmichael numbers congruent to $r$ mod $m$, it suffices to consider Carmichael numbers congruent to $0$ mod $p$ and $r$ mod $m$. Let $r'$ mod $m'$ be the intersection of these two congruence classes. Then $m'=pm$ and $(r',m')=pg$. Setting $g'=(r',m')$, we have the following chain of divisions and equalities: $$(g',2\phi(g'))\mid (g,2\phi(p))(p,2\phi(g))=(g,s)\mid (r,r-1)=1.$$ Also, setting $h'=(\lambda(g'),m')$, we have $h'=[s,h]$. Since $h$ and $s$ both divide $m$, we have that $r'$ is congruent to $r$ modulo both $h$ and $s$. Recalling that $h\mid r-1$ (by hypothesis) and $s\mid r-1$ (essentially by definition), we can then conclude that $h'\mid r'-1$. Finally, $r'/g'\equiv p^{-1}r/g\equiv 1\text{ or }11\lessspacepmod{12}$. Thus $r'\lessspacepmod{m'}$ is Carmichael compatible. 
We conclude that the arithmetic progression $r \lessspacepmod m$ contains a Carmichael compatible subprogression with the property we want.

\end{proof}

\begin{lemma}
In proving Theorem~\ref{main}, we may further assume that the progression $g^{-1} \lessspacepmod{ \lambda(g)}$ contains integers which are $\pm 1$ mod $12$.
\end{lemma}

\begin{proof}
We will assume that $12\mid \lambda(g)$ and $g\equiv 5\text{ or }7\lessspacepmod{12}$, as otherwise there is nothing to prove.
Note that $12\nmid \frac mg$ for the following reason. If this is not the case then $12\mid h\mid r-1$, from which it follows that $\frac{r}{g}\equiv 5\text{ or }7\lessspacepmod{12}$. 
However, this will not be the case assuming the first reduction step has been performed.

Now, first picking $s$ to be either 
4 and 6, take $p$ to be a large prime such that $\frac{p-1}{s}$ has no small prime factors. Then define $r'$, $m'$, $g'$, and $h'$ as before. We have $$(g',2\phi(g'))\mid (g,2\phi(p))(p,2\phi(g))=(g,s)=1$$ since $g$ is relatively prime to $\phi(g)$ (and $12\mid \lambda(g)\mid \phi(g)$). Moreover, $$h'=(\lambda(g'),m')=(\lambda(g),mp)=h$$ since $s\mid \lambda(g)$ so $\lambda(g')\mid \lambda(g)\frac{p-1}{s}$. Now, $h\mid m\mid r'-r$ so we may conclude that $h'\mid r'-1$. Combining this with the afore-mentioned fact that $36\nmid m$, we see that $r'$ mod $m'$ is Carmichael compatible, so we have now completed the reduction process.
\end{proof}

Take $\tilde{m}=\frac{m}{3}$ if $3\mid g$ and $9\nmid m$, and otherwise take $\tilde{m}=m$. Take $a$ mod $q$ to be the intersection of $\frac{r}{g}$ mod $\tilde{m}$ and $g^{-1}$ mod $\lambda(g)$. 
Note that these are compatible: since $(g,\lambda(g))=1$, it suffices to show that $\frac rg\equiv g^{-1}\lessspacepmod{(\frac mg,\lambda(g))}$, which is immediate from the modulus dividing $h$.
We can now assume without loss of generality that the arithmetic progression $a$ mod $q$ contains integers congruent to 1 or 11 mod 12. Indeed, those residues mod 12 which can be represented by integers which are $a$ mod $q$ are the same as those which can be represented by integers congruent to $\frac{r}{g}$ mod $\frac{m}{g}$ and $g^{-1}$ mod $\lambda(g)$, using the fact that $3\mid \tilde{m}$ if and only if $3\mid \frac{m}{g}$.

\begin{lemma}\label{twoprimes}
Let $a\lessspacepmod{q}$ be a reduced congruence class (which is to say that $(a,q)=1$) containing integers congruent to $1$ or $11$ mod $12$, and let $g$ be a positive integer. There exists a square-free integer $P$ with $\frac{\lambda(P)}{2}$ relatively prime to $2q$ such that $a^{-1}P$ has odd order mod $q$ and $gP$ has odd order mod $\frac{\lambda(P)}{2}$.

\end{lemma}
\begin{proof}
Suppose $\frac{p-1}{2}$ is square-free and none of its prime factors are congruent to 1 mod $2^T$, where as before $T$ is a large integer. Then for $m$ relatively prime to $\frac{p-1}{2}$, the equality $$\sum_{\substack{\chi\text{ mod }\frac{p-1}{2}\\ \chi^{2^T}=\chi_0}}\chi(m)=0$$ is equivalent to $m$ having even order mod $\frac{p-1}{2}$.

We claim that for $\ell=2\text{ or }3$, there exist residues $a_1,\ldots,a_\ell$ mod $q$ such that $\prod_{i=1}^{\ell}a_i\equiv a\lessspacepmod{q}$ and moreover $(a_i-1,q)\mid 2$ for each $1\leq i\leq \ell$. Indeed, by the Chinese remainder theorem, such residues exist for any value of $\ell>1$ unless $3\mid q$ or $4\mid q$. In the former case, $\ell$ must be even if $a\equiv 1\lessspacepmod{3}$ and otherwise $\ell$ must be odd. In the latter case, $\ell$ must be even if $a\equiv 1\lessspacepmod 4$ and otherwise $\ell$ must be odd. Therefore, the only scenario in which no such $\ell$ exists is if $12\mid q$ and $a\equiv 5\text{ or }7\lessspacepmod{12}$. However, this is exactly ruled out by hypothesis.

For now, we assume $\ell=3$. Let $n_1$ be a large integer. Let $\mathcal{T}$ be the set of positive integers none of whose prime factors are congruent to 1 mod $2^T$. Let $$\mathcal{P}_1=\{p: p\sim n_1, \frac{p-1}{2}\in\mathcal{T}, p\equiv a_1\lessspacepmod{q}\}.$$ Choose $p_1\in \mathcal{P}_1$. Let $n_2$ be arbitrarily large in terms of $n_1$. Let $$\mathcal{P}_2=\{p: p\sim n_2, \frac{p-1}{2}\in\mathcal{T}, p\equiv a_2\lessspacepmod{q}, p\equiv -1\lessspacepmod{\frac{p_1-1}{2}}\}.$$ Let $n_3$ be arbitrarily large in terms of $n_2$. Take $N_1=\frac{p_1-1}{2}$. Let $A_1$ be a residue mod $N_1$ such that $-gA_1$ is an odd order residue mod $N_1$. Let $N_2$ be the least common multiple of $\{\frac{p-1}{2}: p\in\mathcal{P}_2\}$. Let $A_2$ be a residue mod $N_2$ such that $gp_1A_2$ is an odd order residue mod $N$. Note that $q$, $N_1$, and $N_2$ are pairwise coprime. 
Let $A\lessspacepmod{N}$ be the intersection of the progressions $a_3\lessspacepmod{q}$, $A_1\lessspacepmod{N_1}$, and $A_2\lessspacepmod{N_2}$.

Let $n_3$ be arbitrarily large in terms of $n_2$. Take $$\mathcal{P}_3=\{p: p\sim n_3, \frac{p-1}{2}\in\mathcal{T}, \text{P}^-\big(\frac{p-1}{2}\big)\geq n_3^{1/100}, p\equiv A\lessspacepmod{N}\}.$$ It is worth explaining how we know that $\mathcal{P}_1$, $\mathcal{P}_2$, and $\mathcal{P}_3$ are large. Since $\mathcal{P}_3$ has the most complicated description, we will focus on showing that $\vert\mathcal{P}_3\vert=n_3^{1-o(1)}$. Lemma \ref{rosser}, combined with the Bombieri-Vinogradov theorem, gives us the lower bound we expect (up to a constant) for $$\#\{p: p\sim n_3, \text{P}^-\big(\frac{p-1}{2}\big)\geq n_3^{1/100}, p\equiv A\lessspacepmod{N}\}.$$ On the other hand, for each $q\equiv 1\lessspacepmod{2^T}$, the Selberg sieve (\`{a} la Lemma \ref{Selberg}) gives us the upper bound we expect (up to a constant) for $$\#\{p: p\sim n_3, q\mid \frac{p-1}{2}, \text{P}^-\big(\frac{p-1}{2}\big)\geq n_3^{1/100}, p\equiv A\lessspacepmod{N}\}.$$ Subtracting off these primes for each such $q$ in a suitable range, we obtain $\mathcal{P}_3$. As long as $T$ is chosen large enough in terms of the constants of these sieve results, most of the original primes will remain.

Now, suppose $$\sum_{p_2\in \mathcal{P}_2}\sum_{p_3\in \mathcal{P}_3}\sum_{\substack{\chi\text{ mod }\frac{p_3-1}{2}\\ \chi^{2^T}=\chi_0}}\chi(gp_1p_2)=0.$$ Then $$\sum_{p_2\in \mathcal{P}_2}\sum_{p_3\in \mathcal{P}_3}\sum_{\substack{\chi\text{ mod }\frac{p_3-1}{2}\\ \chi\neq \chi_0\\ \chi^{2^T}=\chi_0}}\chi(gp_1p_2)=-\vert\mathcal{M}\vert\vert \mathcal{Q}\vert,$$ where $\mathcal{M}=\{gp_1p_2: p_2\in\mathcal{P}_2\}$  and $\mathcal{Q}=\{\frac{p-1}{2}: p\in\mathcal{P}_3\}$. By the Cauchy-Schwarz inequality, $$\vert \mathcal{M}\vert^2\vert \mathcal{Q}\vert^2\leq \Big(\sum_{p_3\in\mathcal{P}_3}\sum_{\substack{\chi\text{ mod }\frac{p_3-1}{2}\\ \chi\neq \chi_0\\ \chi^{2^T}=\chi_0}}1\Big)\Big(\sum_{p_3\in\mathcal{P}_3}\sum_{\substack{\chi\text{ mod }\frac{p_3-1}{2}\\ \chi\neq \chi_0\\ \chi^{2^T}=\chi_0}}\Big\vert\sum_{p_2\in\mathcal{P}_2}\chi(gp_1p_2)\Big\vert^2\Big)$$ so $$\vert\mathcal{M}\vert^2\vert\mathcal{Q}\vert\ll \sum_{q\in\mathcal{Q}}\sum_{\substack{\chi\text{ mod }q\\ \chi\neq \chi_0\\ \chi^{2^T}=\chi_0}}\Big\vert\sum_{m\in\mathcal{M}}\chi(m)\Big\vert^2.$$ However, since $\max_{m\in\mathcal{M}}m=\vert\mathcal{Q}\vert^{o(1)}$ and $\vert \mathcal{Q}\vert=\max_{q\in\mathcal{Q}}q^{1-o(1)}$, 
Proposition \ref{elliott} implies that $$\vert\mathcal{M}\vert\vert\mathcal{Q}\vert\gg \sum_{q\in\mathcal{Q}}\sum_{\substack{\chi\text{ mod }q\\ \chi\neq \chi_0\\ \chi^{2^T}=\chi_0}}\Big\vert\sum_{m\in\mathcal{M}}\chi(m)\Big\vert^2.$$ Indeed, non-primitive characters are easily dealt with, using the fact that $\text{P}^-\big(\frac{p-1}{2})\geq n_3^{1/100}$. These contradicting inequalities tell us that there is some $p_2\in\mathcal{P}_2$ and some $p_3\in\mathcal{P}_3$ for which $gP$ is an odd order residue mod $\frac{p_3-1}{2}$, where $P=p_1p_2p_3$. Note also that $gP\equiv gp_1A_2\lessspacepmod{\frac{p_2-1}{2}}$ and $gP\equiv -gA_1\lessspacepmod{\frac{p_1-1}{2}}$, both of which are odd order by assumption. Since $a^{-1}P\equiv 1\lessspacepmod{q}$, this concludes the proof in the case that $\ell=3$. The other case can be dealt with in a strictly easier fashion, omitting all instances of $a_1$, $n_1$, and $p_1$.

\end{proof}
We again take $a$ to be an integer such that \begin{equation}\label{1}a\equiv \frac{r}{g}\lessspacepmod{\tilde{m}}\end{equation} and \begin{equation}\label{2}a\equiv g^{-1}\lessspacepmod{\lambda(g)},\end{equation} while taking $q=[\tilde{m},\lambda(g)]$. We saw before why it can be assumed that being congruent to $a$ mod $q$ is not incompatible with being congruent to 1 or 11 mod 12. The other thing to note is that $a\lessspacepmod{q}$ is a reduced congruence class because $(g,\lambda(g))=1$. We now take $P$ to be an integer supplied by the preceding lemma. Let $a'$ be an integer with \begin{equation}\label{3}a'\equiv aP^{-1}\lessspacepmod{q}\end{equation} and \begin{equation}\label{4}a'\equiv (gP)^{-1}\lessspacepmod{\frac{\lambda(P)}{2}}.\end{equation} Note that these two moduli are relatively prime so such a value of $a'$ does indeed exist. Also take $q'=[q,\lambda(P)/2]$. Note that $a'$ has odd order modulo $q'$. We can therefore find some integer $a_0$ whose square is congruent to $a'$ mod $q'$. Modulo an odd prime power, every reduced residue class which has a square root at all has one which is not 1 modulo the prime, and modulo any power of 2, every reduced residue class which has a square root has one which can be represented by an integer which is 3 mod 4. By the Chinese remainder theorem, we can therefore assume that $(a_0-1,q')\mid 2$. We can also safely assume that $a_0$ is 3 mod 4. Let $q_0=[q',4]$. Then in fact $(a_0-1,q_0)=2$. Take $\ell_0$ to be the order of $a_0$ mod $q_0$, and also take $m_0=gP$.

We now are in a position to prove Theorem \ref{main}. We apply Proposition \ref{workhorse} with these values of $a_0$, $q_0$, $m_0$, and $\ell_0$ to obtain Carmichael seeds. We claim that if $\Pi$ is such a Carmichael seed then $gP\Pi$ is a Carmichael number congruent to $r$ mod $m$. Indeed, $\Pi$ is the product of primes which are $a_0\lessspacepmod{q_0}$, and as $\omega(\Pi)\equiv 2\lessspacepmod{\ell_0}$, it follows that \begin{equation}\label{5}\Pi\equiv a_0^2\equiv a'\lessspacepmod{q'}.\end{equation} By \eqref{3}, $gP\Pi\equiv ga\lessspacepmod{q}$. Then \eqref{1} implies that $gP\Pi$ is congruent to $r$ mod $m$, while \eqref{2} implies that $gP\Pi\equiv 1\lessspacepmod{\lambda(g)}$.
Also, it follows from \eqref{4} and \eqref{5} that $gP\Pi\equiv 1\lessspacepmod{\frac{\lambda(P)}{2}}$, while Proposition \ref{workhorse} tells us that  $gP\Pi\equiv 1\lessspacepmod{\frac{p-1}{2}}$ for each $p\mid \Pi$. We then see that $gP\Pi-1$ is divisible by $\lambda(gP\Pi)$ (using the fact that each prime dividing $P\Pi$ is congruent to 3 mod 4). Thus $gP\Pi$ is indeed a Carmichael number congruent to $r$ mod $m$.

Finally, we return to the question of how small $\frac{\phi(n)}{n}$ can be for a Carmichael number $n$. 

\begin{theorem}
\label{Corollary}
$\liminf_{n\emph{ Carmichael}}\frac{\phi(n)}{n}=0$.
\end{theorem}
\begin{proof}

In \cite{E:1960}, Erd\H os proves that if $q_1=3$ and for $i\ge 2$, $q_i$ is the smallest prime greater than $q_{i-1}$ such that $q_i-1$ is not divisible by any previous prime in the sequence, then the number of primes in the sequence between $x$ and $2x$ grows proportionally to $\frac x{\log x\log\log x}$.  Take $Q$ to be the product of all $q_i$ less than $x$.  Then 
$$-\log \frac{\phi(Q)}Q \gg \sum_{2 < 2^i < x} \frac 1{i \log i}$$
which diverges as $x\to \infty$.  Moreover, for each $x$, Theorem~\ref{main} tells us that there are Carmichael numbers $n$ divisible by $Q$.  For any such $n$, we have that
$\frac{\phi(n)}n \le \frac{\phi(Q)}Q$, so the left hand side can be made arbitrarily small.

\end{proof}
Note that the Carmichael numbers produced in the proof of Theorem~\ref{Corollary} can be taken so that $\frac{\phi(n)}{n}$ is as close to $\frac{\phi(Q)}{Q}$ as we wish.
By taking subproducts of the sequence in \cite{E:1960}, it follows that 
$\{\phi(n)/n\}$ as $n$ ranges over the set of Carmichael numbers is dense in $[0,1]$.
I am grateful to Carl Pomerance for pointing this out, in addition to suggesting a simplification to the proof of Theorem~\ref{Corollary}.

\section*{Acknowledgements}
I would like to thank Andrew Granville, Michael Larsen, Carl Pomerance, and Thomas Wright for many helpful comments and suggestions.

\end{document}